\documentclass[11pt,reqno]{amsart}
\usepackage{color}
\usepackage{amsmath, amssymb, amsthm}

\usepackage{mathrsfs}
\usepackage{mathtools}
\usepackage{aliascnt}
\usepackage{fullpage}
\usepackage[noadjust]{cite}
\usepackage[strict=true]{csquotes}
\usepackage{pifont}
\usepackage{tikz}
\usepackage{tikz-cd}
\usepackage{bbm}
\usepackage[T1]{fontenc}
\usepackage{etoolbox}

\usetikzlibrary{arrows.meta}

\usepackage{environ}
\usepackage{framed}
\usepackage{url}
\usepackage[linesnumbered,ruled,vlined]{algorithm2e}
\usepackage[noend]{algpseudocode}
\usepackage[labelfont=bf]{caption}
\usepackage[framemethod=tikz]{mdframed}
\usepackage{appendix}
\usepackage{graphicx}
\usepackage{tcolorbox}
\usepackage{enumerate}
\usepackage[shortlabels]{enumitem}
\usepackage{physics}
\usepackage[colorlinks=true,allcolors=blue]{hyperref}
\usepackage[noabbrev,capitalize,nameinlink]{cleveref}
\allowdisplaybreaks[1]

\apptocmd{\sloppy}{\hbadness 10000\relax}{}{}

\crefname{equation}{}{}
\crefname{algocf}{Algorithm}{Algorithms}
\crefname{conjecture}{Conjecture}{Conjectures}

\AtBeginEnvironment{appendices}{\crefalias{section}{appendix}}
\makeatletter
\renewcommand{\@bibtitlestyle}{%
  \section*{\refname}%
}
\makeatother
\AtBeginEnvironment{thebibliography}{\scriptsize}

\crefformat{enumi}{#2#1#3}
\crefrangeformat{enumi}{#3#1#4 to~#5#2#6}
\crefmultiformat{enumi}{#2#1#3}{ and~#2#1#3}{, #2#1#3}{ and~#2#1#3}

\numberwithin{equation}{section}
\newtheorem{theorem}{Theorem}[section]

\newaliascnt{proposition}{theorem}
\newtheorem{proposition}[proposition]{Proposition}
\aliascntresetthe{proposition}
\crefname{proposition}{Proposition}{Propositions}

\newaliascnt{lemma}{theorem}
\newtheorem{lemma}[lemma]{Lemma}
\aliascntresetthe{lemma}
\crefname{lemma}{Lemma}{Lemmas}

\newaliascnt{corollary}{theorem}
\newtheorem{corollary}[corollary]{Corollary}
\aliascntresetthe{corollary}
\crefname{corollary}{Corollary}{Corollaries}

\crefname{claim}{Claim}{Claims}

\newtheorem*{question*}{Question}

\theoremstyle{definition}

\newtheorem*{definition*}{Definition}

\theoremstyle{remark}
\newtheorem*{remark}{Remark}

\DeclareQuoteStyle{straight}
  {\textquotedbl}{\textquotedbl}
  {\textquotesingle}{\textquotesingle}
\setquotestyle{straight}
\MakeOuterQuote{"}

\newcommand\HS{\mathrm{HS}}
\DeclarePairedDelimiterX{\inner}[2]{\langle}{\rangle}{#1, #2}

\DeclarePairedDelimiterX{\setcond}[2]{\{}{\}}{#1 \;\delimsize\vert\; #2}
\DeclarePairedDelimiterX{\cond}[2]{(}{)}{#1 \;\delimsize\vert\; #2}
\DeclarePairedDelimiterX{\sqcond}[2]{[}{]}{#1 \;\delimsize\vert\; #2}

\newcommand{\mc}{\mathcal}

\newcommand{\Id}{\mathrm{Id}}

\newcommand{\edit}[1]{{\color{blue}#1}}
\newcommand{\constsub}[1]{_{\edit{#1}}}

\newcommand{\eps}{\varepsilon}

\DeclareMathOperator{\dist}{dist}

\renewcommand{\le}{\leqslant}
\renewcommand{\ge}{\geqslant}

\newcommand\R{\mathbb{R}}

\newcommand\PP{\mathbb{P}}
\newcommand\E{\mathbb{E}}
\newcommand\Var{\mathrm{Var}}

\newcommand*{\deq}{\stackrel{law}{=}}

\usepackage{bm}

\allowdisplaybreaks

\DeclareMathOperator{\Span}{span}

\DeclareMathOperator{\Comp}{Comp}
\DeclareMathOperator{\Incomp}{Incomp}
\DeclareMathOperator{\Sparse}{Sparse}
\DeclareMathOperator{\supp}{supp}
\DeclareMathOperator{\sr}{sr}
\DeclareMathOperator{\diag}{diag}
\def\op{\mathrm{op}}
\newcommand{\normop}[1]{\left\lVert #1 \right\rVert_{\op}}

\title{On the smallest singular value of the product of random and deterministic matrices}

\author[Letwin]{Brayden Letwin}
\address{Department of Mathematics, University of Washington, Seattle, WA 98195}
\email{letwin@uw.edu}

\author[Polavarapu]{Achintya Raya Polavarapu}
\address{School of Mathematics, Georgia Institute of Technology, Atlanta, GA 30332}
\email{apolavarapu6@gatech.edu}

\begin{document}
\hypersetup{pageanchor=false}

\begin{abstract}
Let \(A=(a_{ij})\) be an \(n\times n\) real-valued random matrix with
independent, mean-zero, variance-one entries whose fourth moments are uniformly at most $K$. Suppose that there exists \(\kappa \in (0, 1)\) such that the
entries of \(A\) satisfy
\[
  \max_{i,j}\sup_{u \in \R} \PP(\abs{a_{ij} - u} < 1) \le \kappa.
\]
We prove that there are constants \(c,C>0\), depending only on \(K\) and
\(\kappa\), such that for every fixed invertible \(n\times n\) matrix \(M\)
and every \(\varepsilon\ge0\),
\[
  \PP\!\left(s_{\min}(MA) \le \frac{\varepsilon}{\|M^{-1}\|_{\HS}}\right)
  \le C\varepsilon + e^{-cn}.
\]
In the Gaussian case, we also show that the above estimate is sharp in the sense that $\E[s_{\min}(MA)]\asymp \|M^{-1}\|_{\HS}^{-1}.$
\end{abstract}

\maketitle
\thispagestyle{empty}
\begingroup
\pagestyle{empty}
\endgroup

\setcounter{page}{1}
\hypersetup{pageanchor=true}

\section{Introduction} \label{sec:1}

\medskip\noindent For an $n\times n$ random matrix \(A\) with real-valued entries, its smallest singular value is given by
\[
  s_{\min}(A)=\min_{x\in S^{n-1}}\|Ax\|_2.
\]
For matrices with i.i.d. standard Gaussian entries, it is known that \(s_{\min}(A)\) is of order \(1/\sqrt n\), and precise distributional
estimates in this case go back to Edelman~\cite{Edelman88,Edelman91}; see also Szarek~\cite{Szarek91} for related works. Going beyond the Gaussian case into a broader class of random matrices, Rudelson and Vershynin \cite{RV08,RV08b} proved that if \(A\) has i.i.d. mean-zero and
variance-one sub-gaussian entries then
\begin{equation} \label{eq:1.1}
   \PP\left(s_{\min}(A) \le \frac{\varepsilon}{\sqrt{n}}\right)
  \le C \varepsilon + e^{-cn} 
\end{equation}
for some constants $C, c > 0$ depending only on the sub-gaussian constant of $a_{ij}$.
Their seminal argument follows the now classical compressible/incompressible
decomposition (see \cref{sec:2} for definitions). In their argument, the compressible part of the sphere is
handled by a net argument, and the incompressible part is reduced to a
Littlewood--Offord anti-concentration problem, where the arithmetic structure of
the coefficients becomes relevant; see also \cite{TV09}. Subsequent work on the smallest singular value includes the case of rectangular random matrices~\cite{RV09}, matrices with independent columns~\cite{LPRT05,AGLPT08,GLPTJ17}, sparse matrices~\cite{LR12,BR17}, and
heavy-tailed matrices or matrices whose entries share different distributions~\cite{RT18,Livshyts21,LTV21,DF24}.

\medskip\noindent In this paper, we study the non-asymptotic behaviour of
\(s_{\min}(MA)\), where \(M\) is a fixed invertible \(n\times n\) matrix and
\(A\) is an \(n\times n\) random matrix. The fundamental question is: after
multiplying \(A\) by \(M\), what replaces the classical scale \(1/\sqrt n\)
in lower-tail estimates for \(s_{\min}(A)\)? Even for diagonal matrices \(M\),
and even when \(A\) is Gaussian, this is not obvious a priori. Our main result,
\cref{thm:1.1}, answers this question for a broad class of random matrices and gives the natural generalization of
\cref{eq:1.1}. For a real-valued random variable
\(\xi\) and \(r>0\), let
\(\mathcal{L}(\xi,r)=\sup_{u\in\R}\PP(|\xi-u| < r)\)
denote its L\'evy concentration function.

\begin{theorem} \label{thm:1.1}
Let \(A=(a_{ij})\) be an \(n\times n\) real-valued random matrix with independent entries satisfying for all \(i,j\),
\[
  \E[a_{ij}] = 0,\quad \E[a_{ij}^2] = 1,\quad \E[a_{ij}^4] \le K,
  \quad \mathcal{L}(a_{ij},1)\le \kappa
\]
where \(K\ge1\) and \(\kappa\in(0,1)\). Further, let \(M\) be a fixed real-valued invertible \(n\times n\) matrix.
Then there exist constants \(C,c>0\), depending only on \(K\) and \(\kappa\),
such that for all $\varepsilon \ge 0$
\begin{equation} \label{eq:1.2}
   \PP\Bigl(s_{\min}(MA) \le \frac{\varepsilon}{\|M^{-1}\|_{\HS}}\Bigr)
  \le C \varepsilon + e^{-cn}. 
\end{equation}
\end{theorem}

\medskip\noindent In the Gaussian case, this scale is sharp in expectation. More precisely, \cref{prop:5.expectation-upper}, combined with \cref{thm:1.1},
implies that for matrices $A$ with i.i.d. standard Gaussian entries one has
\[
  \E[s_{\min}(MA)]\asymp \|M^{-1}\|_{\HS}^{-1}.
\]

\begin{remark}
\noindent The diagonal case is already of independent interest. If \(M=\diag(d_1,\dots,d_n)\), then multiplying by \(M\) simply rescales the rows of \(A\). Since singular values are invariant under transposition, the same observation also applies to column-scaled models. In this way, \cref{thm:1.1} yields lower-tail bounds for random matrices with independent entries whose row variances or column variances need not all be the same.
\end{remark}

\medskip\noindent We note that results such as
\cite{LTV21,DF24} treat more general entry distributions than those considered in this paper. Specifically, they don't require a uniform bound on the fourth moments. Despite a series of attempts, we were unable to remove this assumption from \cref{thm:1.1}, and we leave this as an interesting
problem for future work. It is also worth noting that the operator norm of products of random and deterministic matrices $MA$ was studied before; we refer to Vershynin \cite{Vershynin11}, 
who proved that $\E[\normop{MA}] \lesssim \|M\|_{\HS}+\sqrt n\,\|M\|_{\op}$ under different assumptions on $A$. 

\medskip\noindent The paper is organized as follows. In \cref{sec:2} we present notation and give a proof overview. \Cref{sec:3} establishes an estimate that deals with the compressible parts of the sphere. After this, \cref{sec:4} develops the incompressible part of the sphere and concludes the proof of \cref{thm:1.1}. Finally, \cref{sec:5} develops the Gaussian case, proving complementary upper and lower bounds at the same scale and, as a consequence, showing that $\E[s_{\min}(MA)]\asymp \|M^{-1}\|_{\HS}^{-1}.$

\subsection*{Acknowledgements}

\medskip\noindent B.L. and A.P. thank Galyna Livshyts for posing this problem and for her guidance throughout this work.

\section{Preliminaries and proof overview} \label{sec:2}

\medskip\noindent As with many modern results in random matrix theory, the proof follows the compressible and incompressible decomposition argument of Rudelson--Vershynin stated briefly above. Our key new result is \cref{thm:4.4}, which states that if \(\Pi_j\) is the random projection onto the orthogonal complement of the subspace spanned by all columns excluding the \(j\)-th column of \(A\), then
\begin{equation} \label{eq:2.1}
  \E[\Pi_j] \asymp \frac{1}{n}\Id,
\end{equation}
where the comparison is taken in Loewner order. The lower estimate in \cref{eq:2.1} will not be used in the proof of \cref{thm:1.1}; it is included to show sharpness of \cref{thm:4.4}. The upper estimate will be used to prove that
\[
  \E\!\left[\mathbf 1_{\mathcal A_j}\|M^{-\top}z_j\|_2\right]
  \le C\constsub{4.4}\frac{\|M^{-1}\|_{\HS}}{\sqrt n},
\]
where \(\mathcal A_j\) is any event measurable with respect to the \(\sigma\)-algebra generated by all columns except the \(j\)-th column of \(A\). This leads to the incompressible estimate and then to \cref{thm:1.1}. 
\subsection{Notation and conventions} \label{subsec:2.1}

\medskip\noindent We use $c, C > 0$ for positive constants whose values may change from line
to line, and occasionally within the same displayed formula. Constants carrying
a subscript, such as \(c\constsub{3.1}\) or
\(C\constsub{4.4}\), are attached to the indicated result. Unless a different
dependence is explicitly recorded, these constants depend only on \(K\) and
\(\kappa\). We write \(\|\cdot\|_{\op}\) and
\(\|\cdot\|_{\HS}\) for the operator and Hilbert--Schmidt norms,
\(M^{-\top}=(M^{-1})^\top\), and \([n]=\{1,\dots,n\}\). For a nonzero matrix \(B\),
we consider at some point its stable rank
\[
  \sr(B)=\frac{\|B\|_{\HS}^2}{\|B\|_{\op}^2}.
\]
For an event \(\mathcal E\), \(\mathbf 1_{\mathcal E}\) denotes its indicator. We say that a
random vector \(X\in\R^n\) is isotropic if \(\E[X]=0\) and \(\E[XX^\top]=\Id\).
For a vector
\(x\), we write its support as \(\supp(x)=\{i:x_i\ne0\}\). For a subspace \(H\), we write the distance from $x$ to $H$ as
\(\dist(x,H)=\inf_{h\in H}\|x-h\|_2\). For symmetric matrices \(B,C\), we write
\(B\preceq C\) if \(C-B\) is positive semi-definite.

\medskip\noindent Throughout the remainder of the paper, we will assume first that $n$ is sufficiently large (depending on $K$ and $\kappa$). Proving \cref{thm:1.1} for the remaining finitely many $n$ are then handled using known smallest singular value
estimates for when $M = \Id$ along with the inequality
\[
  s_{\min}(MA)\ge s_{\min}(M)\,s_{\min}(A)
  = \frac{s_{\min}(A)}{\|M^{-1}\|_{\op}}.
\]
Since \(\|M^{-1}\|_{\HS}\le \sqrt n\,\|M^{-1}\|_{\op}\), this loses at most a
factor \(\sqrt n\), which is harmless when \(n\) ranges
over a fixed finite set.

\subsection{Overview of the proof} \label{subsec:2.2}

\medskip\noindent Let \(\delta,\rho\in(0,1)\). Following Rudelson--Vershynin, define the set of
\(\delta n\)-sparse unit vectors by
\[
  \Sparse(\delta)=\{x\in S^{n-1}:|\supp(x)|\le \delta n\}
\]
and the corresponding compressible/incompressible subsets of the sphere by
\[
  \Comp(\delta,\rho)
  =
  \{x\in S^{n-1}:\dist(x,\Sparse(\delta))\le \rho\},
  \qquad
  \Incomp(\delta,\rho)
  =
  S^{n-1}\setminus \Comp(\delta,\rho).
\]
\noindent This then gives the compressible/incompressible decomposition of the sphere
\[
  S^{n-1}=\Comp(\delta,\rho)\cup\Incomp(\delta,\rho).
\]
In order to prove \cref{thm:1.1}, we must show that $\norm{MA u}_2$ is not too small for unit vectors $u$. In order to do this, we will prove this for both compressible and incompressible vectors, and then use a union bound. On the compressible vectors \(\Comp(\delta,\rho)\), we can immediately reduce to an estimate for \(A\):
\[
  \|MAu\|_2 \ge \frac{1}{\|M^{-1}\|_{\HS}}\|Au\|_2,
\]
which is immediate by the definition of the operator norm along with $\norm{M^{-1}}_{\op} \le \norm{M^{-1}}_{\HS}$.
The right-hand side depends separately on \(M\) and \(\|Au\|_2\), so we may
apply the bound for the compressible vectors obtained in~\cite[Lemma~5.3]{Livshyts21}.

\medskip\noindent The incompressible side of the argument is more delicate. On
\(\Incomp(\delta,\rho)\), write \(Y_j=Ae_j\) for the columns of $A$, \(H_j=\Span\{Y_k:k\ne j\}\), and
\(\mathcal F_j=\sigma(Y_k:k\ne j)\), so that \(\mathcal F_j\) captures all
columns except \(Y_j\). First, observe that
\[
  u_jMY_j
  =
  MAu-\sum_{k\ne j}u_kMY_k,
\]
and also note that the above sum over $k \neq j$ lies in $MH_j$. Therefore, dividing by $\abs{u_j}$ we may see by definition that
\[
  \dist(MY_j,MH_j)\le \frac{\|MAu\|_2}{|u_j|}.
\] 
The goal is then show that if $\norm{MAu}_2$ is small, then $\dist(MY_j, MH_j)$ is also small and for this, we use a standard result about incompressible vectors, being that, for many $j$ one has $\abs{u_j} \gtrsim 1/\sqrt{n}$. This reduces the proof to showing that $\dist(MY_j, MH_j)$ being small is rare. We may focus on the event \(\dim H_j^\perp=1\), for which we will let \(z_j\) be a unit vector spanning
\(H_j^\perp\). Then we may calculate
\begin{equation} \label{eq:2.2}
\dist(MY_j, MH_j)
  =
  \frac{|\langle z_j, Y_j\rangle|}{\|M^{-\top}z_j\|_2},
\end{equation}
as appearing in \cref{lem:4.3}. Controlling the numerator of \cref{eq:2.2} is then standard, and the new ingredient needed is the estimate for the denominator of \cref{eq:2.2}
\[
  \E\!\left[\mathbf 1_{\mathcal A_j}\|M^{-\top}z_j\|_2\right]
  \le C\constsub{4.4}\frac{\|M^{-1}\|_{\HS}}{\sqrt n},
\]
for every \(\mathcal F_j\)-measurable event \(\mathcal A_j\). This is \cref{thm:4.4}, the rest of the argument after follows easily.

\section{The compressible bound} \label{sec:3}

\medskip\noindent We will use the result \cite[Lemma~5.3]{Livshyts21}, which gives a lower bound for $\norm{Au}_2$ on the set of compressible vectors. One may routinely check that the hypotheses of \cite[Lemma~5.3]{Livshyts21} are satisfied by those of \cref{thm:1.1}.
\begin{lemma}[{\cite[Lemma~5.3]{Livshyts21}}]\label{lem:3.1}
Under the hypotheses of \cref{thm:1.1}, there exist
\(\delta\constsub{3.1},\rho\constsub{3.1}\in(0,1)\) and \(c\constsub{3.1}>0\), depending only on \(\kappa\), such
that
\[
  \PP\!\Bigl(
    \inf_{u \in \Comp(\delta\constsub{3.1},\rho\constsub{3.1})}
    \|Au\|_2
    \le c\constsub{3.1}\sqrt{n}
  \Bigr)
  \le e^{-c\constsub{3.1}n}.
\]
\end{lemma}

\medskip\noindent For the rest of the paper, fix \(\delta\constsub{3.1},\rho\constsub{3.1}\) from \cref{lem:3.1}. The next proposition is the only estimate we need for compressible vectors, and it follows immediately from \cref{lem:3.1} together with the elementary lower bound \(\|Mx\|_2\ge \|M^{-1}\|_{\HS}^{-1}\|x\|_2\).

\begin{proposition} \label{prop:3.2}
Under the hypotheses of \cref{thm:1.1}, with the constants from
\cref{lem:3.1},
\[
  \PP\!\Bigl(
    \inf_{u \in \Comp(\delta\constsub{3.1},\rho\constsub{3.1})} \|MAu\|_2
    \le \frac{c\constsub{3.1}\sqrt{n}}{\|M^{-1}\|_{\HS}}
  \Bigr)
  \le e^{-c\constsub{3.1}n}.
\]
\end{proposition}
\begin{proof}
Since \(\|M^{-1}\|_{\op}\le \|M^{-1}\|_{\HS}\), we have
\[
  \|Mx\|_2 \ge \frac{1}{\|M^{-1}\|_{\op}}\|x\|_2
  \ge \frac{1}{\|M^{-1}\|_{\HS}}\|x\|_2
\]
for every \(x\in\R^n\). Hence for all $u \in \R^n$
\[
  \|MAu\|_2 \ge \frac{1}{\|M^{-1}\|_{\HS}}\|Au\|_2
\]
Therefore
\[
  \left\{
    \inf_{u \in \Comp(\delta\constsub{3.1},\rho\constsub{3.1})}\|MAu\|_2
    \le
    \frac{c\constsub{3.1}\sqrt n}{\|M^{-1}\|_{\HS}}
  \right\}
  \subseteq
  \left\{
    \inf_{u \in \Comp(\delta\constsub{3.1},\rho\constsub{3.1})}\|Au\|_2
    \le
    c\constsub{3.1}\sqrt n
  \right\}.
\]
Apply \cref{lem:3.1} to complete the proof.
\end{proof}

\section{The incompressible bound} \label{sec:4}

\medskip\noindent Now we will focus on controlling the incompressible vectors. Throughout this section, we will write $Y_j = Ae_j$ for the $j$-th column of $A$. The columns $Y_1,\dotsc,Y_n$ are mutually independent and isotropic. For each
$j \in [n]$, define
\[
  W_j = [Y_1\,\cdots\,Y_{j-1}\,Y_{j+1}\,\cdots\,Y_n],
  \qquad
  \mathcal{F}_j = \sigma(Y_k : k \ne j),
  \qquad
  H_j = \Span\{Y_k : k \ne j\}.
\]
Further, we let \(\mathcal{R}_j = \{\dim(H_j^\perp) = 1\}
= \{\operatorname{rank}(W_j)=n-1\}\) be the event that $H_j^{\perp}$ is one-dimensional. Let \(\Pi_j\) denote the orthogonal projection onto \(H_j^\perp\), and fix an \(\mathcal F_j\)-measurable unit vector \(z_j\in H_j^\perp\). On the event \(\mathcal R_j\), this vector is unique up to sign.

\medskip\noindent We will begin by using the property from \cite[Lemma~3.4]{RV08} stating that for every $u \in \Incomp(\delta\constsub{3.1},\rho\constsub{3.1})$ there are many coordinates of $u$ with size $\gtrsim 1/\sqrt{n}$.
\begin{lemma}{\cite[Lemma~3.4]{RV08}} \label{lem:4.1} 
For every $u \in \Incomp(\delta\constsub{3.1},\rho\constsub{3.1})$ there exists a set
$J(u) \subset [n]$ with $|J(u)| \ge \tfrac{1}{2}\rho\constsub{3.1}^2 \delta\constsub{3.1} n$ such that
\[
  \frac{\rho\constsub{3.1}}{\sqrt{2n}} \le |u_j| \le \frac{1}{\sqrt{\delta\constsub{3.1} n}},
  \qquad j \in J(u).
\]
\end{lemma}

\medskip\noindent Next, we will need a lemma stating that $\mc{R}_j$ occurs with exponentially large probability, so that we may reduce to understanding \cref{eq:2.2} on the event $\mc{R}_j$.
\begin{lemma} \label{lem:4.2}
There exists \(c\constsub{4.2}>0\), depending only on \(K\) and \(\kappa\),
such that for every $j \in [n]$,
\[
  \PP(\mathcal{R}_j^c) \le e^{-c\constsub{4.2}n}.
\]
\end{lemma}
\begin{proof}
Let $\widetilde W_j$ be the $(n-1)\times(n-1)$ sub-matrix of \(W_j\) obtained by deleting the last row. Then clearly by inclusions one has
\[
  \PP(\mathcal{R}_j^c)
  \;\le\;
  \PP\!\bigl(s_{\min}(\widetilde{W}_j)=0\bigr).
\]
The entries of \(\widetilde W_j\) satisfy the standing assumptions, so
\cite[Theorem~1.1]{LTV21} at \(\varepsilon=0\) gives a constant
\(c>0\), depending only on \(K\) and \(\kappa\), such that
\[
  \PP\!\bigl(s_{\min}(\widetilde{W}_j)=0\bigr)
  \le
  2e^{-c(n-1)}.
\]
Since throughout we work first with \(n\) sufficiently large, after decreasing
\(c\) and relabeling the constant as \(c\constsub{4.2}\),
\[
  2e^{-c(n-1)}\le e^{-c\constsub{4.2}n}.
\]
The finitely many smaller dimensions are absorbed into the constants as explained in
\cref{subsec:2.1}. This proves the claim.
\end{proof}

\medskip\noindent On the event \(\mathcal R_j\), the distance from $MY_j$ to the subspace $MH_j$ has an explicit one-dimensional formula.
\begin{lemma} \label{lem:4.3}
On the event $\mathcal{R}_j$ one has the following:
\[
  \dist(MY_j, MH_j) = \frac{|\langle z_j, Y_j\rangle|}{\|M^{-\top}z_j\|_2}.
\]
\end{lemma}

\begin{proof}
Recall that $z_j$ spans $H_j^{\perp}$ on the event $\mc{R}_j$. For \(y=Mh\in MH_j\), since $z_j$ is orthogonal to $H_j$ one has that
\[
  \langle M^{-\top}z_j, y\rangle
  = \langle M^{-\top}z_j, Mh\rangle
  = \langle z_j, h\rangle = 0,
\]
so $M^{-\top}z_j \perp MH_j$. Hence
\[
  \dist(MY_j, MH_j)
  = \left|\left\langle
      \frac{M^{-\top}z_j}{\|M^{-\top}z_j\|_2},\, MY_j
    \right\rangle\right|
  = \frac{|\langle z_j, Y_j\rangle|}{\|M^{-\top}z_j\|_2}. \qedhere
\]
\end{proof}

\medskip\noindent As discussed in the proof overview \cref{subsec:2.2} we will need to control the formula appearing in \cref{lem:4.3}. The numerator is standard, however the denominator is new. We control this using \cref{thm:4.4}, which states that
if $\Pi_j$ is defined as above, as the orthogonal projection onto $H_j^{\perp}$ then one has that $\E[\Pi_j]$ is comparable to $\frac{1}{n}\Id$ in Loewner order by absolute constants. We will only need the upper bound for the remainder of the paper, but we believe that this result
is of independent interest, so we also proved a lower bound.

\begin{theorem} \label{thm:4.4}
Under the hypotheses of \cref{thm:1.1}, there exist
\(c\constsub{4.4},C\constsub{4.4} > 0\), depending only on \(K\) and
\(\kappa\), such that for every $j \in [n]$,
\[
  \frac{c\constsub{4.4}}{n}\Id
  \preceq
  \mathbb{E}[\Pi_j]
  \preceq
  \frac{C\constsub{4.4}}{n}\,\Id.
\]
Consequently, for every $\mathcal{F}_j$-measurable event $\mathcal{A}_j$,
\[
  \mathbb{E}\!\left[
    \mathbf{1}_{\mathcal{A}_j}\|M^{-\top}z_j\|_2
  \right]
  \le C\constsub{4.4}\frac{\|M^{-1}\|_{\HS}}{\sqrt{n}}.
\]
\end{theorem}
\begin{proof}
\textbf{Upper bound.}
Fix \(j\in[n]\). We reveal the columns other than the original \(Y_j\) one at a
time. To lighten notation, we will assume without loss of generality that $j = n$ and thus deal with these columns as \(Y_1,\dots,Y_{n-1}\) for
the duration. Set
\[
  P_0=\Id,\qquad G_0=\Omega,\qquad
  \mathcal G_0=\{\emptyset,\Omega\},\qquad A_0=\Id.
\]
For \(1\le m\le n-1\), let \(P_m\) be the orthogonal projection onto
\(\Span(Y_1,\dots,Y_m)^\perp\), and set
\[
  G_m=\{\operatorname{rank}([Y_1,\dots,Y_m])=m\},
  \qquad
  \mathcal G_m=\sigma(Y_1,\dots,Y_m),
  \qquad
  A_m=\E[P_m\mathbf 1_{G_m}].
\]
With these definitions, $G_m$ is the event that the first $m$ columns are linearly independent, and $A_m$ is the average projection of $P_m$, but only on the event $G_m$. Observe that by construction one has
\(P_{n-1}=\Pi_j\). For \(0\le m\le n-2\), set
\(\xi_m=P_mY_{m+1}\), the component of \(Y_{m+1}\) lying in
\(\Span(Y_1,\dots,Y_m)^\perp\). One may then verify that
\begin{equation}
  \label{eq:4.1}
  P_{m+1} = P_m - \frac{\xi_m \xi_m^\top}{\|\xi_m\|_2^2}\,\mathbf{1}_{\{\xi_m \ne 0\}},
  \qquad G_{m+1} = G_m \cap \{\xi_m \ne 0\}.
\end{equation}
Define the \(\mathcal G_m\)-measurable matrix
\begin{equation}
  \label{eq:4.2}
  B_m
  =
  \mathbb{E}\!\left[
    \frac{\xi_m\xi_m^\top}{\|\xi_m\|_2^2}\,
    \mathbf{1}_{\{\xi_m \ne 0\}}
    \,\middle|\, \mathcal G_m
  \right].
\end{equation}
Since \(Y_{m+1}\) is independent of \(\mathcal G_m\), conditioning on
\(\mathcal G_m\) gives through \eqref{eq:4.1} and \eqref{eq:4.2} that
\[
\begin{aligned}
  \mathbb{E}[P_{m+1}\mathbf{1}_{G_{m+1}} \mid \mathcal{G}_m]
  &=
  \mathbf{1}_{G_m}
  \left(
    P_m\,\PP(\xi_m\ne0\mid\mathcal G_m)-B_m
  \right) \\
  &\preceq \mathbf{1}_{G_m}(P_m - B_m),
\end{aligned}
\]
where last inequality only follows from
\(\PP(\xi_m\ne0\mid \mathcal G_m)\le1\).
Therefore, by the total law, we may uncondition and observe that
\begin{equation}
  \label{eq:4.3}
  A_{m+1}
  \preceq
  \mathbb{E}[\mathbf{1}_{G_m}(P_m - B_m)].
\end{equation}

\medskip\noindent We now lower bound \(B_m\), uniformly in \(m\). To do this, we work conditionally on \(\mathcal G_m\), so \(P_m\) is
fixed, and take \(x\in S^{n-1}\cap\operatorname{Im}(P_m)\). By
\eqref{eq:4.2},
\[
  \langle x, B_mx \rangle
  = \mathbb{E}\!\left[
    \frac{\langle x,\xi_m \rangle^2}{\|\xi_m\|_2^2}\,
    \mathbf{1}_{\{\xi_m \ne 0\}}
    \,\middle|\, \mathcal G_m
  \right].
\]
Since \(x\in\operatorname{Im}(P_m)\), one has
\(\langle x,\xi_m\rangle=\langle x,Y_{m+1}\rangle\). Let
\(d=\operatorname{rank}(P_m)\) and
\(\mathcal A_d=\{|\|\xi_m\|_2^2-d|\le d^{3/4}\}\). Since
\(\|\xi_m\|_2^2=Y_{m+1}^\top P_mY_{m+1}\), and since \(P_m\) is fixed after
conditioning on \(\mathcal G_m\) while \(Y_{m+1}\) remains isotropic,
\[
  \E[\|\xi_m\|_2^2\mid\mathcal G_m]
  =
  \tr\!\left(P_m\,\E[Y_{m+1}Y_{m+1}^\top\mid\mathcal G_m]\right)
  =
  \tr(P_m\Id)
  =
  \tr(P_m)
  =
  d.
\]
Writing \(P_m=(p_{ij})\), a computation gives
\[
  \mathbb{E}[\|\xi_m\|_2^4\mid\mathcal G_m]
  =
  (\tr P_m)^2 + 2\tr(P_m^2)
  + \sum_{i=1}^n (\mathbb{E}[(Y_{m+1})_i^4]-3)p_{ii}^2.
\]
Since \(\tr(P_m)=\tr(P_m^2)=d\), \(\sum_i p_{ii}^2\le d\), and \(\E[(Y_{m+1})_i^4]\le K\), this gives
\[
  \mathbb{E}[\|\xi_m\|_2^4\mid\mathcal G_m] \le d^2 + (K+2)d,
\]
so the preceding two displays and Chebyshev give
\[
  \Var(\|\xi_m\|_2^2\mid\mathcal G_m)\le (K+2)d,
  \qquad
  \PP(\mathcal A_d^c\mid\mathcal G_m)\le \frac{K+2}{\sqrt d}.
\]
For \(d\ge2\), on \(\mathcal A_d\) we have \(\|\xi_m\|_2^2>0\), and hence
\[
  \langle x, B_mx \rangle
  \ge \frac{1}{d(1 + d^{-1/4})}\,
  \mathbb{E}\!\left[
    \langle x,Y_{m+1} \rangle^2 \mathbf{1}_{\mathcal{A}_d}
    \,\middle|\,\mathcal G_m
  \right].
\]
Also
\[
  \E[\langle x,Y_{m+1}\rangle^2\mid\mathcal G_m]=1,\qquad
  \E[\langle x,Y_{m+1}\rangle^4\mid\mathcal G_m]
  =
  3+\sum_{i=1}^n(\E[(Y_{m+1})_i^4]-3)x_i^4
  \le K+3,
\]
and therefore Cauchy--Schwarz gives
\[
  \E[\langle x,Y_{m+1}\rangle^2\mathbf 1_{\mathcal A_d^c}\mid\mathcal G_m]
  \le \sqrt{(K+2)(K+3)}\,d^{-1/4}.
\]
Thus, setting
\[
  \eta_d =
  1-\frac{1-\sqrt{(K+2)(K+3)}\,d^{-1/4}}{1+d^{-1/4}},
\]
we have \(0\le \eta_d\le (1+\sqrt{(K+2)(K+3)})d^{-1/4}\). Combining the preceding estimates, we get
\[
  \langle x, B_mx \rangle \ge \frac{1 - \eta_d}{d},
\]
and hence, using that $B_m = P_mB_mP_m$ one then has that
\begin{equation}
  \label{eq:4.4}
  B_m\succeq \frac{1-\eta_d}{d}\,P_m.
\end{equation}
On \(G_m\), the projection \(P_m\) has rank \(d=n-m\). Applying
\eqref{eq:4.4} in \eqref{eq:4.3}, and using the definition of \(A_m\), yields
\[
  A_{m+1}
  \preceq
  \left(1 - \frac{1-\eta_{n-m}}{n-m}\right)A_m.
\]
Now we will iterate this bound to compare $\Id$ and $A_{n-1}$, we will begin by choosing
\[
  d_0
  =
  \max\!\left\{2,
  \left\lceil \bigl(2+2\sqrt{(K+2)(K+3)}\bigr)^4\right\rceil
  \right\},
\]
so that \(0\le\eta_d\le 1/2\) for all \(d\ge d_0\).
For \(d<d_0\), we use the trivial estimate \(A_{m+1}\preceq A_m\), and so, iterating from $A_0 = \Id$ we see
\[
  A_{n-1} \preceq \prod_{d=d_0}^{n}
  \left(1 - \frac{1-\eta_d}{d}\right) \Id.
\]
Since $\log(1-x) \le -x$,
\[
  \begin{aligned}
  \log\prod_{d=d_0}^n\!\left(1 - \frac{1-\eta_d}{d}\right)
  &\le
  -\sum_{d=d_0}^n\frac{1}{d}
  +
  (1+\sqrt{(K+2)(K+3)})\sum_{d=d_0}^\infty d^{-5/4} \\
  &\le
  -\log\!\left(\frac{n+1}{d_0}\right)
  +
  (1+\sqrt{(K+2)(K+3)})\sum_{d=d_0}^\infty d^{-5/4},
  \end{aligned}
\]
so after exponentiating we an promptly observe that there exists a constant $C\constsub{4.4} > 0$ such that
\[
  A_{n-1}\preceq \frac{C\constsub{4.4}}{n}\Id.
\]
The finitely many cases \(n<d_0\) are absorbed by increasing \(C\constsub{4.4}\).
Since \(G_{n-1}=\mathcal R_j\), the last display controls
\(\E[\Pi_j\mathbf 1_{\mathcal R_j}]\). The complement of \(\mathcal R_j\) is
exponentially unlikely by \cref{lem:4.2}, and \(0\preceq P_{n-1}\preceq\Id\);
therefore after abusing notation with absolute constants
\[
  \mathbb{E}[\Pi_j]
  = A_{n-1} + \mathbb{E}[P_{n-1}\mathbf{1}_{G_{n-1}^c}]
  \preceq \frac{C\constsub{4.4}}{n}\,\Id
  + e^{-c\constsub{4.2}n}\Id
  \preceq \frac{C\constsub{4.4}}{n}\,\Id.
\]

\medskip\noindent\textbf{Proving the implication of the upper bound.}
Now we will prove the consequently part of \cref{thm:4.4}. Since \(z_j\in H_j^\perp\) is a unit vector and \(\Pi_j\) is the orthogonal projection onto
\(H_j^\perp\) we have that
\[
  \mathbf 1_{\mathcal A_j}z_jz_j^\top\preceq \Pi_j.
\]
Thus, by the upper bound just proved,
\[
  \mathbb{E}\!\left[\mathbf{1}_{\mathcal{A}_j}\|M^{-\top}z_j\|_2^2\right]
  = \tr\!\left(M^{-1}M^{-\top}
    \mathbb{E}[\mathbf{1}_{\mathcal{A}_j}z_jz_j^\top]\right)
  \le \frac{C\constsub{4.4}}{n}\,\tr(M^{-1}M^{-\top})
  = \frac{C\constsub{4.4}\|M^{-1}\|_{\HS}^2}{n}.
\]
By Cauchy--Schwarz,
\[
  \mathbb{E}\!\left[\mathbf{1}_{\mathcal{A}_j}\|M^{-\top}z_j\|_2\right]
  \le \mathbb{E}\!\left[
    \mathbf{1}_{\mathcal{A}_j}\|M^{-\top}z_j\|_2^2
  \right]^{1/2}
  \le \sqrt{C\constsub{4.4}}\frac{\|M^{-1}\|_{\HS}}{\sqrt{n}}
  \le C\constsub{4.4}\frac{\|M^{-1}\|_{\HS}}{\sqrt{n}}.
\]
This completes the proof of the upper bound and the latter estimate of \cref{thm:4.4}.

\medskip\noindent\textbf{Lower bound.}
The lower bound uses the same strategy of revealing columns one at a time, but in the opposite direction. We
must show that the new revealed column cannot, on average, remove almost all of the
current orthogonal complement. The case \(n=1\) is immediate, since then \(\Pi_j=\Id\), so
assume \(n\ge2\). We keep the same enumeration and projections \(P_m\) as
above, but no longer multiply by the indicators \(\mathbf 1_{G_m}\). Conditionally on
\(\mathcal G_m\), let \(\xi_m=P_mY_{m+1}\). Then, for
\(0\le m\le n-2\),
\[
   P_{m+1}
   =
   P_m
   -
   \frac{\xi_m \xi_m^\top}{\|\xi_m\|_2^2}\mathbf 1_{\{\xi_m\ne0\}}.
\]
Using \eqref{eq:4.2},
\[
   \E[P_{m+1}\mid \mathcal G_m]
   =
   P_m-B_m.
\]

\medskip\noindent We first prove an upper bound for $B_m$. We claim that there is an integer
\(d_0\ge2\) and a constant \(C>0\), depending only on \(K\) and \(\kappa\),
and a constant \(\alpha>0\), depending only on \(K\), such that,
conditionally on \(\mathcal G_m\), if
\(d=\operatorname{rank}(P_m)\ge2\), then
\begin{equation}
  \label{eq:4.5}
   B_m\preceq
   \begin{cases}
      (1-\alpha)P_m, & 2\le d<d_0,\\[4pt]
      \dfrac{1+Cd^{-1/4}}{d}P_m, & d\ge d_0.
   \end{cases}
\end{equation}

\medskip\noindent To prove \cref{eq:4.5}, let us first work conditionally on \(\mathcal G_m\). Take
\(x\in S^{n-1}\cap\operatorname{Im}(P_m)\). Since
\(\langle x,\xi_m\rangle=\langle x,Y_{m+1}\rangle\), recall that \eqref{eq:4.2} gives
\[
   \langle x,B_mx\rangle
   =
   \E\left[
      \frac{\langle x,Y_{m+1}\rangle^2}{\|\xi_m\|_2^2}
      \mathbf 1_{\{\xi_m\ne0\}}
      \,\middle|\, \mathcal G_m
   \right].
\]
We will split into the case where $d \ge d_0$ (large ranks), and into the case where $2 \le d < d_0$ (small ranks). Let us begin first in the case where $d$ is large.

\medskip\noindent\textit{Large ranks.}
Consider first the case of large \(d\). By \cite[Corollary~1.4]{RV15}, there is a constant
\(c>0\), depending only on \(\kappa\), such that, for \(\xi_m\) one has that
\[
   \PP(\|\xi_m\|_2\le c\sqrt d\mid\mathcal G_m)
   \le e^{-cd}.
\]
After decreasing \(c\), assume \(c<1\). Decompose according to
\[
   \mathcal E_0=\{\|\xi_m\|_2^2\le c^2d\},\qquad
   \mathcal E_1=\{c^2d<\|\xi_m\|_2^2<d-d^{3/4}\},\qquad
   \mathcal E_2=\{\|\xi_m\|_2^2\ge d-d^{3/4}\}.
\]
For all sufficiently large \(d\), these events cover the whole probability
space. On \(\mathcal E_2\),
\[
   \frac{\langle x,Y_{m+1}\rangle^2}{\|\xi_m\|_2^2}
   \le
   \frac{\langle x,Y_{m+1}\rangle^2}{d-d^{3/4}},
\]
and since \(\E[\langle x,Y_{m+1}\rangle^2\mid\mathcal G_m]=1\),
\[
   \E\left[
      \frac{\langle x,Y_{m+1}\rangle^2}{\|\xi_m\|_2^2}
      \mathbf 1_{\mathcal E_2}
      \,\middle|\,\mathcal G_m
   \right]
   \le
   \frac{1}{d-d^{3/4}}
   \le
   \frac{1+2d^{-1/4}}{d}
\]
for all sufficiently large \(d\). On \(\mathcal E_1\),
\[
   \frac{\langle x,Y_{m+1}\rangle^2}{\|\xi_m\|_2^2}
   \le
   \frac{\langle x,Y_{m+1}\rangle^2}{c^2d}.
\]
By the same computation used in the upper bound, since
\(\|\xi_m\|_2^2=Y_{m+1}^\top P_mY_{m+1}\), we have
\[
   \E[\|\xi_m\|_2^2\mid\mathcal G_m]=d,
   \qquad
   \operatorname{Var}(\|\xi_m\|_2^2\mid\mathcal G_m)\le (K+2)d.
\]
Hence
\[
   \PP(\|\xi_m\|_2^2<d-d^{3/4}\mid\mathcal G_m)
   \le
   (K+2)d^{-1/2}.
\]
Moreover,
\[
   \E[\langle x,Y_{m+1}\rangle^4\mid\mathcal G_m]
   =
   3+\sum_{i=1}^n(\E[(Y_{m+1})_i^4]-3)x_i^4
   \le K+3.
\]
Let \(C>0\) be a constant depending only on \(K\) and \(\kappa\), large
enough for the following estimates. Therefore, by Cauchy--Schwarz,
\[
\begin{aligned}
   \E\left[
      \frac{\langle x,Y_{m+1}\rangle^2}{\|\xi_m\|_2^2}
      \mathbf 1_{\mathcal E_1}
      \,\middle|\,\mathcal G_m
   \right]
   &\le
   \frac{1}{c^2d}
   \E\left[
      \langle x,Y_{m+1}\rangle^2
      \mathbf 1_{\{\|\xi_m\|_2^2<d-d^{3/4}\}}
      \,\middle|\,\mathcal G_m
   \right]  \\
   &\le
   \frac{1}{c^2d}
   (\E[\langle x,Y_{m+1}\rangle^4\mid\mathcal G_m])^{1/2}
   \PP(\|\xi_m\|_2^2<d-d^{3/4}\mid\mathcal G_m)^{1/2}       \\
   &\le
   Cd^{-5/4}.
\end{aligned}
\]
Finally, on \(\mathcal E_0\), we use
\(\langle x,Y_{m+1}\rangle^2\le \|\xi_m\|_2^2\), which holds because
\(x\in\operatorname{Im}(P_m)\). Thus
\[
   \E\left[
      \frac{\langle x,Y_{m+1}\rangle^2}{\|\xi_m\|_2^2}
      \mathbf 1_{\mathcal E_0\cap\{\xi_m\ne0\}}
      \,\middle|\,\mathcal G_m
   \right]
   \le
   \PP(\mathcal E_0\mid\mathcal G_m)
   \le e^{-cd}
   \le Cd^{-5/4}
\]
for all sufficiently large \(d\). Combining the estimates over
\(\mathcal E_0,\mathcal E_1,\mathcal E_2\) gives
\[
   \langle x,B_mx\rangle
   \le
   \frac{1+Cd^{-1/4}}{d}.
\]
Since this holds for every unit vector \(x\in\operatorname{Im}(P_m)\), and
\(B_m=P_mB_mP_m\), we get
\[
   B_m\preceq \frac{1+Cd^{-1/4}}{d}P_m
\]
for all \(d\ge d_0\), after increasing \(d_0\) if necessary.

\medskip\noindent\textit{Small ranks.}
It remains to treat the finitely many ranks \(2\le d<d_0\). Fix
\(x\in S^{n-1}\cap\operatorname{Im}(P_m)\). Since \(\operatorname{rank}(P_m)\ge2\),
choose \(y\in S^{n-1}\cap\operatorname{Im}(P_m)\) with \(y\perp x\). Since
\(\langle y,\xi_m\rangle=\langle y,Y_{m+1}\rangle\), we have
\[
   \|\xi_m\|_2^2
   \ge
   \langle x,Y_{m+1}\rangle^2+\langle y,Y_{m+1}\rangle^2,
\]
and therefore
\[
   \frac{\langle x,Y_{m+1}\rangle^2}{\|\xi_m\|_2^2}
   \mathbf 1_{\{\xi_m\ne0\}}
   \le
   1-
   \frac{\langle y,Y_{m+1}\rangle^2}
   {\langle x,Y_{m+1}\rangle^2+\langle y,Y_{m+1}\rangle^2}
   \mathbf 1_{\{\langle x,Y_{m+1}\rangle^2
   +\langle y,Y_{m+1}\rangle^2>0\}}.
\]
Since
\(\E[\langle y,Y_{m+1}\rangle^2\mid\mathcal G_m]=1\) and
\(\E[\langle y,Y_{m+1}\rangle^4\mid\mathcal G_m]\le K+3\),
Paley--Zygmund gives
\[
   \PP(|\langle y,Y_{m+1}\rangle|\ge 1/\sqrt2\mid\mathcal G_m)
   \ge
   \frac{1}{4(K+3)}.
\]
Choose \(R^2=8(K+3)\). Since
\(\E[\langle x,Y_{m+1}\rangle^2\mid\mathcal G_m]=1\),
\[
   \PP(|\langle x,Y_{m+1}\rangle|>R\mid\mathcal G_m)\le \frac{1}{R^2}
   =
   \frac{1}{8(K+3)}.
\]
Hence
\[
   \PP(|\langle y,Y_{m+1}\rangle|\ge 1/\sqrt2,\,
   |\langle x,Y_{m+1}\rangle|\le R\mid\mathcal G_m)
   \ge
   \frac{1}{8(K+3)}.
\]
On this event,
\[
   \frac{\langle y,Y_{m+1}\rangle^2}
   {\langle x,Y_{m+1}\rangle^2+\langle y,Y_{m+1}\rangle^2}
   \ge
   \frac{1/2}{R^2+1/2}
   =
   \frac{1}{2R^2+1}.
\]
Thus
\[
   \E\left[
      \frac{\langle y,Y_{m+1}\rangle^2}
      {\langle x,Y_{m+1}\rangle^2+\langle y,Y_{m+1}\rangle^2}
      \mathbf 1_{\{\langle x,Y_{m+1}\rangle^2
      +\langle y,Y_{m+1}\rangle^2>0\}}
      \,\middle|\,\mathcal G_m
   \right]
   \ge
   \frac{1}{8(K+3)(2R^2+1)}.
\]
Set
\[
   \alpha_0=\frac{1}{8(K+3)(2R^2+1)}
   \qquad\text{and}\qquad
   \alpha=\min\{\alpha_0,1/2\}.
\]
The preceding displays give
\[
   \langle x,B_mx\rangle\le 1-\alpha.
\]
Since this holds for every unit \(x\in\operatorname{Im}(P_m)\), we get
\[
   B_m\preceq (1-\alpha)P_m
\]
for \(2\le d<d_0\). Together, the large-rank and small-rank cases prove
\cref{eq:4.5}.

\medskip\noindent\textit{Iterating the lower bound.}
For \(0\le m\le n-2\), the subspace onto
which \(P_m\) projects has dimension at least \(n-m\). In particular,
\(\operatorname{rank}(P_m)\ge n-m\ge2\). Increase \(d_0\), if necessary, so
that the function
\[
   f(d)=\frac{1+Cd^{-1/4}}{d}
\]
is decreasing for \(d\ge d_0\), and so that \(f(d)\le1/2\) for \(d\ge d_0\).
Define
\[
   \gamma_r
   =
   \begin{cases}
      \alpha, & 2\le r<d_0,\\[4pt]
      1-\dfrac{1+Cr^{-1/4}}{r}, & r\ge d_0.
   \end{cases}
\]
Since \(\operatorname{rank}(P_m)\ge n-m\), \cref{eq:4.5} gives
\[
   B_m\preceq (1-\gamma_{n-m})P_m.
\]
Indeed, if \(n-m\ge d_0\), this follows from the monotonicity of \(f\); if
\(n-m<d_0\), then either \(\operatorname{rank}(P_m)<d_0\), or else
\(B_m\preceq f(\operatorname{rank}(P_m))P_m\preceq \tfrac12P_m
\preceq (1-\alpha)P_m\). Consequently,
\begin{equation}
  \label{eq:4.6}
   \E[P_{m+1}\mid\mathcal G_m]
   =
   P_m-B_m
   \succeq
   \gamma_{n-m}P_m.
\end{equation}
Taking expectations in \eqref{eq:4.6} and iterating from \(m=0\) to
\(m=n-2\), with \(P_0=\Id\),
we obtain
\[
   \E[\Pi_j]
   =
   \E[P_{n-1}]
   \succeq
   \left(\prod_{r=2}^{n}\gamma_r\right)\Id.
\]

\medskip\noindent
It remains to estimate the product. The finitely many
factors with \(2\le r<d_0\) contribute a positive constant depending only on
\(K\) and \(\kappa\). For \(r\ge d_0\), put
\[
   x_r=\frac{1+Cr^{-1/4}}{r}.
\]
By our choice of \(d_0\), \(0<x_r\le1/2\), and therefore
\[
   \log(1-x_r)\ge -x_r-2x_r^2.
\]
Hence
\[
\begin{aligned}
   \sum_{r=d_0}^n \log(1-x_r)
   &\ge
   -\sum_{r=d_0}^n \frac1r
   -C\sum_{r=d_0}^\infty r^{-5/4}
   -2\sum_{r=d_0}^\infty x_r^2 \\
   &\ge
   -\log n - C,
\end{aligned}
\]
where \(C>0\) depends only on \(K\) and \(\kappa\). Thus
there is a constant \(c>0\), depending only on \(K\) and \(\kappa\), such that
\[
   \prod_{r=d_0}^n
   \left(1-\frac{1+Cr^{-1/4}}{r}\right)
   \ge
   \frac{c}{n}.
\]
Absorbing the finite product, and also the finitely many cases \(n<d_0\), we
get, after decreasing \(c\) if necessary,
\begin{equation}
  \label{eq:4.7}
   \prod_{r=2}^n\gamma_r\ge \frac{c}{n}.
\end{equation}
Therefore, by the preceding iteration and \eqref{eq:4.7},
\[
   \E[\Pi_j]\succeq \frac{c\constsub{4.4}}{n}\Id.
\]
This proves the lower bound and completes the proof of the theorem.
\end{proof}

\begin{remark}
\cref{thm:4.4} is the only place where the fourth moment assumption is required on the entries of $A$. It would be interesting to try to remove this assumption.
\end{remark}

\medskip\noindent Now that \cref{thm:4.4} is complete, it remains to control the numerator in \cref{lem:4.3}.
After \cref{thm:4.4}, the denominator \(\|M^{-\top}z_j\|_2\) is under control
on average, but this is useful only if the numerator \(|\langle z_j,Y_j\rangle|\)
is not too often exceptionally small. The role of \cref{lem:4.5} is to provide exactly this estimate,
through a direct application of the machinery from \cite{LTV21}.

\begin{lemma} \label{lem:4.5}
There exist constants \(C\constsub{4.5},c\constsub{4.5}>0\), depending only on \(K\) and \(\kappa\), and \(\mathcal F_j\)-measurable events \(\mathcal E_j\subseteq \mathcal R_j\) with
\[
  \PP(\mathcal E_j^c)\le e^{-c\constsub{4.5}n},
\]
such that on \(\mathcal E_j\),
\[
  \PP\!\bigl(|\langle z_j,Y_j\rangle|\le \varepsilon \,\big|\, \mathcal F_j\bigr)
  \le C\constsub{4.5}\bigl(\varepsilon + e^{-c\constsub{4.5}n}\bigr)
  \qquad\text{for all } \varepsilon\ge0.
\]
\end{lemma}

\begin{proof}
Fix \(j\in[n]\), and let \(W_j\) be the matrix obtained from \(A\) by deleting the \(j\)-th column, so that \(\mathcal F_j=\sigma(W_j)\). Since \(Y_j\) is
independent of \(\mathcal F_j\), has independent entries, satisfies
\[
  \mathcal L\bigl((Y_j)_i,1\bigr)\le \kappa
  \qquad\text{and}\qquad
  \mathbb E (Y_j)_i^2=1
\]
for every \(i\), and moreover
\[
  \mathbb E\|Y_j\|_2^2=n,
\]
the hypotheses of \cite[Lemma~2.1, Proposition~5.3, and Lemma~2.5]{LTV21} are satisfied uniformly in \(j\). Hence, for all sufficiently large \(n\), there
exist constants \(C_1,c_1>0\), depending only on \(K\) and \(\kappa\), and an \(\mathcal F_j\)-measurable event \(\mathcal G_j\) such that
\[
  \PP(\mathcal G_j^c)\le e^{-c_1n}+2^{-n/2},
\]
and such that on \(\mathcal G_j\),
\[
  \PP\!\bigl(|\langle z_j,Y_j\rangle|\le \varepsilon \,\big|\, \mathcal F_j\bigr)
  \le C_1\varepsilon + C_1 e^{-c_1n}
  \qquad\text{for all }\varepsilon\ge 0.
\]
The finitely many smaller dimensions are absorbed into the constants, as in
\cref{subsec:2.1}.

\medskip\noindent We now intersect with \(\mathcal R_j\), since only on \(\mathcal R_j\) is \(H_j^\perp\) one-dimensional, so that \(z_j\) is the
normal vector appearing in \cref{lem:4.3}. Set
\[
  \mathcal E_j=\mathcal G_j\cap \mathcal R_j.
\]
By \cref{lem:4.2}, after decreasing the exponential rate if necessary, there exists \(c_2>0\), depending only on \(K\) and \(\kappa\), such that
\[
  \PP(\mathcal E_j^c)\le e^{-c_2n}.
\]
Since \(\mathcal E_j\subseteq \mathcal G_j\), the preceding conditional estimate remains valid on \(\mathcal E_j\). After enlarging \(C_1\), decreasing
\(c_2\), and relabeling constants, we obtain
\[
  \PP\!\bigl(|\langle z_j,Y_j\rangle|\le \varepsilon \,\big|\, \mathcal F_j\bigr)
  \le C\constsub{4.5}\bigl(\varepsilon+e^{-c\constsub{4.5}n}\bigr)
  \qquad\text{for all }\varepsilon\ge 0
\]
on \(\mathcal E_j\), and also
\[
  \PP(\mathcal E_j^c)\le e^{-c\constsub{4.5}n}.
\]
This proves the lemma.
\end{proof}

\medskip\noindent The numerator estimate and averaged denominator bound finally give \cref{cor:4.6}, which shows that $\dist(MY_j, MH_j)$ being small is rare on the event $\mc{R}_j$.
\begin{corollary} \label{cor:4.6}
There exist constants \(C\constsub{4.6},c\constsub{4.6}>0\), depending only on
\(K\) and \(\kappa\), such that for every \(j\in[n]\) and every
\(\varepsilon\ge0\),
\[
  \PP\!\Bigl(
    \dist(MY_j,MH_j)\le \frac{\varepsilon\sqrt n}{\|M^{-1}\|_{\HS}},\;
    \mathcal R_j
  \Bigr)
  \le C\constsub{4.6}\bigl(\varepsilon+e^{-c\constsub{4.6}n}\bigr).
\]
\end{corollary}
\begin{proof}
Since \(\mathcal E_j\subseteq\mathcal R_j\),
\[
  \PP\!\Bigl(
    \dist(MY_j,MH_j)\le \frac{\varepsilon\sqrt n}{\|M^{-1}\|_{\HS}},\;
    \mathcal R_j
  \Bigr)
  \le
  \PP(\mathcal E_j^c)
  +
  \PP\!\Bigl(
    \dist(MY_j,MH_j)\le \frac{\varepsilon\sqrt n}{\|M^{-1}\|_{\HS}},\;
    \mathcal E_j
  \Bigr).
\]
On \(\mathcal E_j\),
\[
  \dist(MY_j,MH_j)
  =
  \frac{|\langle z_j,Y_j\rangle|}{\|M^{-\top}z_j\|_2}.
\]
Thus, by \cref{lem:4.5},
\begin{align*}
  &\PP\!\Bigl(
    \dist(MY_j,MH_j)\le \frac{\varepsilon\sqrt n}{\|M^{-1}\|_{\HS}},\;
    \mathcal E_j
  \Bigr) \\
  &\qquad=
  \E\!\left[
    \mathbf 1_{\mathcal E_j}
    \PP\!\Bigl(
      |\langle z_j,Y_j\rangle|
      \le
      \frac{\varepsilon\sqrt n}{\|M^{-1}\|_{\HS}}\,
      \|M^{-\top}z_j\|_2
      \,\Big|\, \mathcal F_j
    \Bigr)
  \right] \\
  &\qquad\le
  C\constsub{4.5}
  \frac{\varepsilon\sqrt n}{\|M^{-1}\|_{\HS}}\,
  \E\!\left[\mathbf 1_{\mathcal E_j}\|M^{-\top}z_j\|_2\right]
  +
  C\constsub{4.5}e^{-c\constsub{4.5}n} \\
  &\qquad\le
  C\constsub{4.4}C\constsub{4.5}\varepsilon
  +
  C\constsub{4.5}e^{-c\constsub{4.5}n}.
\end{align*}
Renaming then yields constants \(C\constsub{4.6}\) and \(c\constsub{4.6}\), proving the claim.
\end{proof}

\medskip\noindent We now pass from \cref{cor:4.6} to getting a full bound on the contribution added to the smallest singular value of $MA$ from the incompressible part of the sphere.

\begin{proposition} \label{prop:4.7}
There exist constants \(C\constsub{4.7},c\constsub{4.7}>0\), depending only on
\(K\) and \(\kappa\), such that for every $\varepsilon \ge 0$,
\[
  \PP\!\Bigl(
    \exists\, u \in \Incomp(\delta\constsub{3.1},\rho\constsub{3.1})
    : \|MAu\|_2 \le \frac{\varepsilon}{\|M^{-1}\|_{\HS}}
  \Bigr)
  \le C\constsub{4.7}\bigl(\varepsilon + e^{-c\constsub{4.7}n}\bigr).
\]
\end{proposition}
\begin{proof}
Define
\[
\begin{aligned}
  \mathcal{A}_\varepsilon
  &=
  \Bigl\{\exists\,u\in\Incomp(\delta\constsub{3.1},\rho\constsub{3.1}):
  \|MAu\|_2\le \varepsilon/\|M^{-1}\|_{\HS}\Bigr\},\\
  N_\varepsilon
  &=
  \sum_{j=1}^n\Bigl(
    \mathbf{1}_{\{\dist(MY_j,MH_j)\le
      (\sqrt2/\rho\constsub{3.1})\varepsilon\sqrt{n}/\|M^{-1}\|_{\HS},\,
      \mathcal{R}_j\}}
    +\mathbf{1}_{\mathcal{R}_j^c}
  \Bigr).
\end{aligned}
\]
If \(u\) witnesses \(\mathcal A_\varepsilon\), \cref{lem:4.1} gives \(J(u)\subset[n]\) with
\[
  |J(u)| \ge \frac12 \rho\constsub{3.1}^2 \delta\constsub{3.1} n
\]
and
\[
  |u_j| \ge \frac{\rho\constsub{3.1}}{\sqrt{2n}},
  \qquad
  j \in J(u).
\]
For \(j\in J(u)\), put \(w=\sum_{k\ne j}u_kMY_k\in MH_j\). Then
\[
  \dist(u_jMY_j, MH_j)
  = \inf_{h\in MH_j}\|MAu - w - h\|_2
  = \dist(MAu, MH_j)
  \le \|MAu\|_2 \le \frac{\varepsilon}{\|M^{-1}\|_{\HS}}.
\]
Dividing by $\abs{u_j}$, we may see that
\[
  \dist(MY_j, MH_j)
  = \frac{\dist(u_jMY_j,MH_j)}{|u_j|}
  \le
  \frac{\sqrt2}{\rho\constsub{3.1}}
  \frac{\varepsilon\sqrt{n}}{\|M^{-1}\|_{\HS}}.
\]
Hence, on \(\mathcal A_\varepsilon\),
\[
  N_\varepsilon \ge |J(u)| \ge \frac12 \rho\constsub{3.1}^2 \delta\constsub{3.1} n.
\]
Therefore
\[
  \PP(\mathcal{A}_\varepsilon)
  \le
  \frac{2}{\rho\constsub{3.1}^2 \delta\constsub{3.1} n}\,\mathbb{E}[N_\varepsilon].
\]
By \cref{cor:4.6} and \cref{lem:4.2},
\[
  \mathbb{E}[N_\varepsilon]
  \le
  nC\constsub{4.6}
  \left(
    \frac{\sqrt2}{\rho\constsub{3.1}}\varepsilon
    +
    e^{-c\constsub{4.6}n}
  \right)
  +
  ne^{-c\constsub{4.2}n}.
 \]
Since \(\delta\constsub{3.1}\) and \(\rho\constsub{3.1}\) are fixed,
\[
  \PP(\mathcal{A}_\varepsilon)
  \le
  C\constsub{4.7}\bigl(\varepsilon + e^{-c\constsub{4.7}n}\bigr). \qedhere
\]
\end{proof}

\medskip\noindent The theorem now follows by combining the compressible estimate of \cref{prop:3.2} with the incompressible estimate of \cref{prop:4.7}.
\begin{proof}[Proof of \cref{thm:1.1}]
For \(0\le\varepsilon\le c\constsub{3.1}\),
\cref{prop:3.2} gives, since \(n\ge 1\),
\[
  \PP\!\Bigl(
    \inf_{u\in\Comp(\delta\constsub{3.1},\rho\constsub{3.1})}
    \|MAu\|_2 \le \frac{\varepsilon}{\|M^{-1}\|_{\HS}}
  \Bigr)
  \le e^{-c\constsub{3.1}n}.
\]
Decompose
\[
  S^{n-1}
  =\Comp(\delta\constsub{3.1},\rho\constsub{3.1})
   \cup \Incomp(\delta\constsub{3.1},\rho\constsub{3.1}).
\]
For the incompressible part, \cref{prop:4.7} gives
\[
  \PP\!\Bigl(\exists\,u\in\Incomp(\delta\constsub{3.1},\rho\constsub{3.1}):
  \|MAu\|_2\le \varepsilon/\|M^{-1}\|_{\HS}\Bigr)
  \le C\constsub{4.7}\bigl(\varepsilon + e^{-c\constsub{4.7}n}\bigr).
\]
A union bound over the two cases gives constants \(C,c>0\), depending
only on \(K\) and \(\kappa\), such that
\[
  \PP\!\left(s_{\min}(MA)\le \frac{\varepsilon}{\|M^{-1}\|_{\HS}}\right)
  \le C\varepsilon + Ce^{-cn}
\]
for \(0\le\varepsilon\le c\constsub{3.1}\). Since \(n\) is sufficiently large,
after decreasing \(c\) we have \(Ce^{-cn}\le e^{-c\constsub{1.1}n}\).
Renaming constants gives
\[
  \PP\!\left(s_{\min}(MA)\le \frac{\varepsilon}{\|M^{-1}\|_{\HS}}\right)
  \le C\constsub{1.1}\varepsilon + e^{-c\constsub{1.1}n}
\]
for \(0\le\varepsilon\le c\constsub{3.1}\). The finitely many smaller dimensions
are handled as in \cref{subsec:2.1}. For \(\varepsilon>c\constsub{3.1}\), the
bound follows after enlarging \(C\constsub{1.1}\), since
\(\PP(\cdot)\le 1\le c\constsub{3.1}^{-1}\varepsilon\).
\end{proof}

\section{Upper bound for Gaussian entries} \label{sec:5}

\medskip\noindent In this section, we show that when $A$ has i.i.d. Gaussian entries, \cref{thm:1.1} is sharp in the sense that $1/\|M^{-1}\|_{\HS}$ is the correct scale. That is,
\[
\E[s_{\min}(MA)] \asymp \frac{1}{\norm{M^{-1}}_{\HS}}.
\]
Unlike the lower bound in
\cref{thm:1.1}, the high-probability upper bound depends on the effective
number of directions in which \(M^{-1}\) is large. This effective dimension is
measured by the stable rank \(\sr(M^{-1})\). We also provide an example at the end of \cref{sec:5} showing that the dependence on the stable rank is not an artifact of our proofs.

\medskip\noindent Assume here that \(A\) has i.i.d.\ \(\mathcal N(0,1)\) entries. We use a classical Gaussian smallest singular value estimate, going back to Edelman~\cite{Edelman88,Edelman91}: there exist absolute constants \(C,c>0\) such that
\begin{equation}
\label{eq:5.1}
\PP\!\left(s_{\min}(A)>\frac{t}{\sqrt n}\right)
\;\le\;
  Ce^{-ct^2},
  \qquad t\ge 1.
\end{equation}
The main result for this section is \cref{thm:5.1}, which gives a bound analogous to \cref{eq:5.1}.
\begin{theorem} \label{thm:5.1}
Let $A$ be an $n\times n$ random matrix with i.i.d.\
$\mathcal{N}(0,1)$ entries, and let $M$ be any fixed invertible $n\times n$
matrix. Write
\[
  \sr(M^{-1})
  \;=\;
  \frac{\|M^{-1}\|_{\HS}^2}{\|M^{-1}\|_{\mathrm{op}}^2}.
\]
There exist absolute constants \(C\constsub{5.1},c\constsub{5.1}>0\) such that
for all \(n\ge 1\) and every \(t\ge 1\),
\begin{equation}
  \label{eq:5.2}
  \PP\!\left(
    s_{\min}(MA)\le \frac{C\constsub{5.1}t}{\|M^{-1}\|_{\HS}}
  \right)
  \;\ge\;
  1
  -
  C\constsub{5.1}e^{-c\constsub{5.1}t^2}
  -
  C\constsub{5.1}e^{-c\constsub{5.1}\sr(M^{-1})}.
\end{equation}
\end{theorem}

\medskip\noindent The proof of \cref{thm:5.1} combines \eqref{eq:5.1} with a lower-tail estimate
for \(\|M^{-\top}u_n(A)\|_2\), where \(u_n(A)\) is the random left singular
vector corresponding to \(s_{\min}(A)\). The next lemma provides exactly this
second result, which is an immediate corollary of the Hanson--Wright inequality~\cite[Theorem~6.2.1]{Vershynin18}.

\begin{lemma} \label{lem:5.2}
Let $N$ be a fixed $n\times n$ positive semi-definite matrix with $N\ne 0$, and
let $v\sim\mathrm{Unif}(S^{n-1})$ be independent of $N$. Then there exist
absolute constants \(C\constsub{5.2},c\constsub{5.2}>0\) such that
\begin{equation}
  \label{eq:5.3}
  \PP\!\left(v^\top N v\;\le\;\frac{\tr(N)}{4n}\right)
  \;\le\;
  C\constsub{5.2}
  \exp\!\left(
    -c\constsub{5.2}\frac{\tr(N)}{\|N\|_{\mathrm{op}}}
  \right).
\end{equation}
\end{lemma}

\begin{proof}
Let \(g\sim\mathcal N(0,\Id)\). We may write \(v\deq g/\|g\|_2\), and
\[
  v^\top Nv \;=\; \frac{g^\top Ng}{\|g\|_2^2}.
\]
Set \(\lambda=\tfrac12\tr(N)\). Hanson--Wright~\cite[Theorem~6.2.1]{Vershynin18} gives
\[
  \PP\!\left(g^\top Ng \le \tr(N) - \lambda\right)
  \;\le\;
  \exp\!\left(
    -c
    \min\!\left\{
      \frac{\lambda^2}{\|N\|_{\HS}^2},\;\frac{\lambda}{\|N\|_{\mathrm{op}}}
    \right\}
  \right).
\]
Using \(\|N\|_{\op}\tr(N)\ge \|N\|_{\HS}^2\) we promptly see
\[
  \frac{\lambda^2}{\|N\|_{\HS}^2}
  \;\ge\;
  \frac{\tr(N)}{4\|N\|_{\mathrm{op}}},
  \qquad
  \frac{\lambda}{\|N\|_{\op}}
  \ge
  \frac{\tr(N)}{4\|N\|_{\op}},
\]
and thus we may obtain that
\[
  \PP\!\left(g^\top Ng \le \tfrac{1}{2}\,\tr(N)\right)
  \;\le\;
  \exp\!\left(
    -c\frac{\tr(N)}{\|N\|_{\mathrm{op}}}
  \right).
\]
Also, the standard tail bound for the norm of a Gaussian random vector gives
\[
  \PP\!\left(\|g\|_2^2 > 2n\right) \;\le\; C e^{-c n}.
\]
On the complement of these two events,
\[
  v^\top Nv \;=\; \frac{g^\top Ng}{\|g\|_2^2}
  \;\ge\;
  \frac{\tr(N)}{4n}.
\]
Since \(\tr(N)\le n\|N\|_{\op}\), we see
\[
  \PP\!\left(v^\top Nv \le \frac{\tr(N)}{4n}\right)
  \;\le\;
  C\constsub{5.2}
  \exp\!\left(
    -c\constsub{5.2}\frac{\tr(N)}{\|N\|_{\op}}
  \right),
\]
for some appropriate constants \(C\constsub{5.2}\) and \(c\constsub{5.2}\).
\end{proof}

\begin{proof}[Proof of \cref{thm:5.1}]
Set \(N=M^{-1}M^{-\top}\). If \(u=u_n(A)\) is a unit left singular vector for \(s_{\min}(A)\), let
\[
  y=\frac{M^{-\top}u}{\|M^{-\top}u\|_2}.
\]
Then
\[
  s_{\min}(MA)
  \le
  \|(MA)^\top y\|_2
  \;=\;
  \frac{s_{\min}(A)}{\|M^{-\top}u\|_2}.
\]

\medskip\noindent By Gaussian rotational invariance we have that $u_n(A)$ is uniform on the sphere. Furthermore, note that $u_n(A)$ is independent of $s_{\min}(A)$. Let us define two events
\[
  \mathcal{A}_t
  \;=\;
  \left\{s_{\min}(A) \le \frac{t}{\sqrt{n}}\right\},
  \qquad
  \mathcal{D}
  \;=\;
  \left\{
    \|M^{-\top}u_n(A)\|_2 \ge \frac{\|M^{-1}\|_{\HS}}{2\sqrt{n}}
  \right\}.
\]
On \(\mathcal A_t\cap\mathcal D\) we have that
\[
  s_{\min}(MA)
  \;\le\;
  \frac{t/\sqrt{n}}{\|M^{-1}\|_{\HS}/(2\sqrt{n})}
  \;=\;
  \frac{2t}{\|M^{-1}\|_{\HS}}.
\]
Moreover, since \(\tr(N)=\|M^{-1}\|_{\HS}^2\) and \(\|N\|_{\op}=\|M^{-1}\|_{\op}^2\), \cref{lem:5.2} gives
\[
  \PP(\mathcal D^c)
  \;\le\;
  C\constsub{5.2}e^{-c\constsub{5.2}\sr(M^{-1})}.
\]
Together with \eqref{eq:5.1},
\[
  \PP\!\left(s_{\min}(MA) > \frac{2t}{\|M^{-1}\|_{\HS}}\right)
  \;\le\;
  \PP(\mathcal{A}_t^c) + \PP(\mathcal{D}^c)
  \;\le\;
  Ce^{-ct^2}
  +
  C\constsub{5.2}e^{-c\constsub{5.2}\sr(M^{-1})},
\]
which proves \eqref{eq:5.2} after taking \(C\constsub{5.1}\ge
\max\{2,C,C\constsub{5.2}\}\) and
\(c\constsub{5.1}\le\min\{c,c\constsub{5.2}\}\).
\end{proof}

\medskip\noindent The term
\(C\constsub{5.1}e^{-c\constsub{5.1}\sr(M^{-1})}\) in \cref{thm:5.1}
prevents one from directly integrating the tail estimate in \(t\). Fortunately however, by
combining that theorem with a cruder estimate, we can prove that $\E[s_{\min}(MA)] \asymp \norm{M^{-1}}_{\HS}^{-1}$. The next proposition establishes the upper bound, while the lower bound is immediate from \cref{thm:1.1}.

\begin{proposition}
\label{prop:5.expectation-upper}
Let \(A\) be an \(n\times n\) random matrix with i.i.d.\ \(\mathcal N(0,1)\) entries, and let \(M\) be a fixed invertible \(n\times n\) matrix. Then
\[
  \E[s_{\min}(MA)]
  \le \frac{C\constsub{5.3}}{\|M^{-1}\|_{\HS}},
\]
where \(C\constsub{5.3}>0\) is an absolute constant.
\end{proposition}

\begin{proof}
Let \(c>0\) denote an absolute constant, whose value may decrease from line
to line within this proof. First note that
\[
  s_{\min}(M) =
  \frac{\sqrt{\sr(M^{-1})}}{\|M^{-1}\|_{\HS}}.
\]

\medskip\noindent By the singular value decomposition and the rotational invariance of the Gaussian, it suffices to consider the case where \(M=D=\diag(d_1,\dots,d_n)\) is diagonal, and where \(d_1=s_{\min}(M)\). Let \(G_1,\dots,G_n\) denote the rows of a standard Gaussian matrix \(G\). Since the smallest singular value is bounded above by the distance of any row to the span of the remaining rows,
\[
  s_{\min}(DG)
  \le
  \dist(d_1G_1,\Span\{d_2G_2,\dots,d_nG_n\}).
\]
The scalars \(d_2,\dots,d_n\) are nonzero as we assumed that $D$ was invertible, so the span on the right is \(\Span\{G_2,\dots,G_n\}\). Conditionally on \(G_2,\dots,G_n\), this is an \((n-1)\)-dimensional subspace almost surely, and therefore if $z$ is the unit normal spanning \(\Span\{G_2,\dots,G_n\}^{\perp}\) then one has
\[
  \dist(G_1,\Span\{G_2,\dots,G_n\}) = \abs{\inner{G_1}{z}}
  \sim |g|,
  \qquad g\sim \mathcal N(0,1).
\]
Hence, for every \(t\ge0\),
\begin{equation}
\label{eq:row-tail-upper}
  \PP(s_{\min}(MA)>t)
  \le
  2\exp\left(-c\frac{t^2}{s_{\min}(M)^2}\right)
  =
  2\exp\left(-c\frac{t^2\|M^{-1}\|_{\HS}^2}{\sr(M^{-1})}\right).
\end{equation}

\medskip\noindent On the other hand, Theorem~\ref{thm:5.1} gives, for every
\(t\ge1\),
\begin{equation}
\label{eq:main-tail-for-expectation}
  \PP\!\left(
    s_{\min}(MA)> \frac{C\constsub{5.1}t}{\|M^{-1}\|_{\HS}}
  \right)
  \le
  C\constsub{5.1}e^{-c\constsub{5.1}t^2}
  +
  C\constsub{5.1}e^{-c\constsub{5.1}\sr(M^{-1})}.
\end{equation}

\medskip\noindent Using the layer-cake formula and splitting the integral at
\(t=\sr(M^{-1})\), we get
\[
\begin{aligned}
  \E[s_{\min}(MA)]
  &=
  \frac{C\constsub{5.1}}{\|M^{-1}\|_{\HS}}
  \int_0^\infty
  \PP\!\left(
    s_{\min}(MA)> \frac{C\constsub{5.1}t}{\|M^{-1}\|_{\HS}}
  \right)
  \,dt  \\
  &\le
  \frac{C\constsub{5.1}}{\|M^{-1}\|_{\HS}}
  \Biggl[
    1
    +
    \int_1^{\sr(M^{-1})}
      C\constsub{5.1}e^{-c\constsub{5.1}t^2}
    \,dt \\
  &\qquad\qquad
    +
    \int_1^{\sr(M^{-1})}
      C\constsub{5.1}e^{-c\constsub{5.1}\sr(M^{-1})}
    \,dt
    +
    \int_{\sr(M^{-1})}^\infty
      2\exp\left(
        -cC\constsub{5.1}^{\,2}t
      \right)
    \,dt
  \Biggr].
\end{aligned}
\]
One may verify that the integrals are all at most absolute constants, giving the desired result.
\end{proof}

\medskip\noindent Now, using \cref{thm:1.1} we may conclude that when $A$ is a Gaussian matrix with standard i.i.d. entries that
\[
\E s_{\min}(MA) \asymp \norm{M^{-1}}_{\HS}^{-1}
\]

\medskip\noindent For the remainder of the paper, we will also show that one may translate another classical result about Gaussian random matrices to the case of $MA$. \cref{prop:5.4} shows that the lower tail for the smallest singular value of $MA$ is at least linear in $\varepsilon$ for small enough $\varepsilon$, when $A$ is a standard Gaussian matrix.

\begin{proposition} \label{prop:5.4}
Let \(A\) be an \(n\times n\) random matrix with i.i.d.\ \(\mathcal N(0,1)\)
entries, and let \(M\) be a fixed invertible \(n\times n\) matrix. Then there
exist absolute constants \(c\constsub{5.4},\eps\constsub{5.4}>0\) such that for
every \(n\ge2\) and every \(0\le\varepsilon\le \eps\constsub{5.4}\),
\[
  \PP\!\left(
    s_{\min}(MA)\le \frac{\varepsilon}{\|M^{-1}\|_{\HS}}
  \right)
  \ge c\constsub{5.4}\varepsilon.
\]
\end{proposition}

\begin{proof}
Set \(N=M^{-1}M^{-\top}\), and let \(u=u_n(A)\) be a unit left singular vector
corresponding to \(s_{\min}(A)\). As in the proof of \cref{thm:5.1},
\[
  s_{\min}(MA)\le \frac{s_{\min}(A)}{\|M^{-\top}u\|_2}.
\]
\medskip\noindent
By Gaussian rotational invariance again we see that $u$ is uniform on the sphere and independent of $s_{\min}(A)$. To ease notation, let us define
\[
  X=\|M^{-\top}u\|_2^2=u^\top N u.
\]
Then
\[
  \E[X]=\frac{\tr(N)}{n}=\frac{\|M^{-1}\|_{\HS}^2}{n}.
\]
\medskip\noindent
If \(\lambda_1,\dots,\lambda_n\) are the eigenvalues of \(N\), then by
standard fourth-moment identities we see
\[
\begin{aligned}
  \E[X^2]
  &=
  \sum_{i=1}^n \lambda_i^2\E[u_i^4]
  +
  \sum_{i\ne j}\lambda_i\lambda_j\E[u_i^2u_j^2] \\
  &=
  \frac{3\sum_{i=1}^n\lambda_i^2+\sum_{i\ne j}\lambda_i\lambda_j}{n(n+2)}
  =
  \frac{(\tr N)^2+2\tr(N^2)}{n(n+2)} \\
  &\le
  \frac{3(\tr N)^2}{n(n+2)}
  \le
  3(\E[X])^2.
\end{aligned}
\]
\medskip\noindent
Therefore Paley--Zygmund gives
\[
  \PP\!\left(X\ge \frac12\,\E[X]\right)
  \ge
  \frac{(1-1/2)^2(\E[X])^2}{\E[X^2]}
  \ge
  \frac{1}{12}.
\]
\medskip\noindent
Equivalently,
\[
  \PP\!\left(
    \|M^{-\top}u\|_2\ge \frac{\|M^{-1}\|_{\HS}}{\sqrt{2n}}
  \right)
  \ge
  \frac{1}{12}.
\]
\medskip\noindent
Hence
\[
\begin{aligned}
  \PP\!\left(
    s_{\min}(MA)\le \frac{\varepsilon}{\|M^{-1}\|_{\HS}}
  \right)
  &\ge
  \PP\!\left(
    \|M^{-\top}u\|_2\ge \frac{\|M^{-1}\|_{\HS}}{\sqrt{2n}},
    \,
    s_{\min}(A)\le \frac{\varepsilon}{\sqrt{2n}}
  \right) \\
  &=
  \PP\!\left(
    \|M^{-\top}u\|_2\ge \frac{\|M^{-1}\|_{\HS}}{\sqrt{2n}}
  \right)
  \PP\!\left(
    s_{\min}(A)\le \frac{\varepsilon}{\sqrt{2n}}
  \right) \\
  &\ge
  \frac{1}{12}\,
  \PP\!\left(
    s_{\min}(A)\le \frac{\varepsilon}{\sqrt{2n}}
  \right).
\end{aligned}
\]
\medskip\noindent
Estimates going back to
Edelman \cite{Edelman88,Edelman91}, imply that there exists
an absolute constant \(c>0\) such that
\[
  \PP\!\left(
    s_{\min}(A)\le \frac{t}{\sqrt n}
  \right)
  \ge
  c t,
  \qquad 0\le t\le c.
\]
Applying this with \(t=\varepsilon/\sqrt2\) gives the claim, after changing the
absolute constants.
\end{proof}

\medskip\noindent \Cref{prop:5.expectation-upper,prop:5.4} show that \(\|M^{-1}\|_{\HS}^{-1}\) remains the correct
Gaussian scale both for the expectation and for the lower tail at small
\(\varepsilon\), even when \(\sr(M^{-1})\) is small. \cref{thm:5.1} gives a much stronger high-probability statement when
\(\sr(M^{-1})\) is large, because then the residual term
\(e^{-c\constsub{5.1}\sr(M^{-1})}\) doesn't matter. When \(\sr(M^{-1})\) is
bounded by an absolute constant,
however, one should not expect an overwhelming-probability upper bound at the
same scale. \cref{prop:5.5} gives an example showing that the stable rank term is sometimes necessary: if $1 \le r \le n - 1$ then
there are examples of $M$ for which \(r\le \sr(M^{-1})\le 2r\) and the event
\[
  s_{\min}(MA)\le \frac{C}{\|M^{-1}\|_{\HS}}
\]
fails with probability at least \(e^{-cr}\).

\begin{proposition} \label{prop:5.5}
Fix \(C_0>0\). There exists \(K\constsub{5.5}>0\), depending only on \(C_0\), such that for every \(n\ge2\) and every \(1\le r\le n-1\), there is an invertible diagonal matrix \(M\) satisfying
\[
  r\le \sr(M^{-1})\le 2r
\]
and
\[
  \PP\!\left(s_{\min}(MA)> \frac{C_0}{\|M^{-1}\|_{\HS}}\right)
  \ge e^{-K\constsub{5.5}r}.
\]
\end{proposition}

\begin{proof}
Choose \(a\constsub{5.5}>0\) small enough depending only on \(C_0\), and put
\[
  M^{-1}
  =
  \diag(\underbrace{n/a\constsub{5.5},\dots,n/a\constsub{5.5}}_{r\text{ times}},
  1,\dots,1).
\]
Then
\[
  \sr(M^{-1})
  =
  \frac{r n^2/a\constsub{5.5}^2+n-r}{n^2/a\constsub{5.5}^2}
  =
  r+\frac{a\constsub{5.5}^2(n-r)}{n^2}.
\]
Taking \(a\constsub{5.5}\le1\), we have
\[
  r\le \sr(M^{-1})\le 2r.
\]
Moreover,
\[
  \frac{1}{\|M^{-1}\|_{\HS}}
  =
  \frac{1}{\sqrt{r n^2/a\constsub{5.5}^2+n-r}}
  \le
  \frac{a\constsub{5.5}}{n\sqrt r}.
\]

\medskip\noindent Let \(X_i^\top\) denote the \(i\)-th row of \(A\). Then
\[
  (MA)^\top
  =
  \bigl[
    (a\constsub{5.5}/n)X_1,\dots,(a\constsub{5.5}/n)X_r,
    X_{r+1},\dots,X_n
  \bigr].
\]
Let
\[
  H=\Span\{X_{r+1},\dots,X_n\}.
\]
Almost surely, \(\dim H=n-r\). Choose an orthonormal basis \(Z\) for
\(H^\perp\) and an orthonormal basis \(U\) for \(H\). Let \(O=[Z\ U]\).
Multiplying on the left by \(O^\top\), which preserves singular values, gives us the following representation
\[
  \begin{pmatrix}
    (a\constsub{5.5}/n)G_{11} & 0 \\
    (a\constsub{5.5}/n)G_{21} & G_{22}
  \end{pmatrix},
\]
where \(G_{11}\) is an \(r\times r\) standard Gaussian matrix, \(G_{21}\) is an
\((n-r)\times r\) standard Gaussian matrix, and the singular values of
\(G_{22}\) are precisely the nonzero singular values of the \(n\times(n-r)\)
Gaussian matrix
\[
  [X_{r+1}\,\cdots\,X_n].
\]
Conditionally on \(X_{r+1},\dots,X_n\), the matrices \(G_{11}\) and \(G_{21}\)
are independent standard Gaussian blocks and are independent of \(G_{22}\).

\medskip\noindent Let
\[
  \mathcal E
  =
  \left\{
    s_{\min}(G_{11})\ge \frac{4(C_0+1)}{\sqrt r}
  \right\}
  \cap
  \left\{
    \|G_{21}\|_{\op}\le 2\sqrt n
  \right\}
  \cap
  \left\{
    s_{\min}(G_{22})\ge
    \frac{8(C_0+1)a\constsub{5.5}}{\sqrt n}
  \right\}.
\]

\medskip\noindent By Edelman's estimates for square Gaussian matrices
\cite{Edelman88,Edelman91}, there is \(C=C(C_0)>0\) such that
\[
  \PP\!\left(
    s_{\min}(G_{11})\ge \frac{4(C_0+1)}{\sqrt r}
  \right)
  \ge e^{-Cr}.
\]
A standard estimate, see
\cite{Vershynin18}, gives an absolute constant \(c>0\) such
that
\[
  \PP\!\left(\|G_{21}\|_{\op}\le 2\sqrt n\right)
  \ge 1-e^{-cn}.
\]
Since \(G_{22}\) has the nonzero singular values of an \(n\times(n-r)\)
Gaussian matrix, rectangular smallest singular value estimates (see \cite{RV09})
implies that, for \(a\constsub{5.5}\) small enough depending only on \(C_0\),
there is an absolute constant \(c>0\) such that
\[
  \PP\!\left(
    s_{\min}(G_{22})\ge
    \frac{8(C_0+1)a\constsub{5.5}}{\sqrt n}
  \right)
  \ge c.
\]
Therefore, after increasing \(K\constsub{5.5}\) depending only on \(C_0\),
\[
  \PP(\mathcal E)\ge e^{-K\constsub{5.5}r}.
\]
We now lower bound \(s_{\min}(MA)\) on \(\mathcal E\). Let
\(x=(x_0,x')\in S^{n-1}\), where \(x_0\in\R^r\) and
\(x'\in\R^{n-r}\). If \(\|x'\|_2\le 1/2\), then
\(\|x_0\|_2\ge \sqrt3/2\), and hence
\[
  \|(MA)^\top x\|_2
  \ge
  \frac{a\constsub{5.5}}{n}\|G_{11}x_0\|_2
  \ge
  \frac{2(C_0+1)a\constsub{5.5}}{n\sqrt r}.
\]
If \(\|x'\|_2\ge 1/2\), then
\[
  \|(MA)^\top x\|_2
  \ge
  \|G_{22}x'\|_2
  -
  \frac{a\constsub{5.5}}{n}\|G_{21}\|_{\op}\|x_0\|_2.
\]
On \(\mathcal E\), this gives
\[
  \|(MA)^\top x\|_2
  \ge
  \frac{8(C_0+1)a\constsub{5.5}}{\sqrt n}\cdot\frac12
  -
  \frac{a\constsub{5.5}}{n}\cdot 2\sqrt n
  =
  \frac{(4C_0+2)a\constsub{5.5}}{\sqrt n}
  \ge
  \frac{2C_0a\constsub{5.5}}{n\sqrt r}.
\]
Thus, in both cases,
\[
  \|(MA)^\top x\|_2
  \ge
  \frac{2C_0a\constsub{5.5}}{n\sqrt r}.
\]
Taking the infimum over \(x\in S^{n-1}\), we obtain
\[
  s_{\min}(MA)
  =
  s_{\min}((MA)^\top)
  \ge
  \frac{2C_0a\constsub{5.5}}{n\sqrt r}.
\]
Since
\[
  \frac{1}{\|M^{-1}\|_{\HS}}
  \le \frac{a\constsub{5.5}}{n\sqrt r},
\]
we have
\[
  s_{\min}(MA)> \frac{C_0}{\|M^{-1}\|_{\HS}}
\]
on \(\mathcal E\). Therefore
\[
  \PP\!\left(s_{\min}(MA)> \frac{C_0}{\|M^{-1}\|_{\HS}}\right)
  \ge
  \PP(\mathcal E)
  \ge
  e^{-K\constsub{5.5}r}.
\]
This proves the proposition.
\end{proof}

\bibliographystyle{amsplain}
\bibliography{main}

@article {AGLPT08,
    AUTHOR = {Adamczak, Rados\l aw and Gu\'edon, Olivier and Litvak,
              Alexander and Pajor, Alain and Tomczak-Jaegermann, Nicole},
     TITLE = {Smallest singular value of random matrices with independent
              columns},
   JOURNAL = {C. R. Math. Acad. Sci. Paris},
  FJOURNAL = {Comptes Rendus Math\'ematique. Acad\'emie des Sciences. Paris},
    VOLUME = {346},
      YEAR = {2008},
    NUMBER = {15-16},
     PAGES = {853--856},
      ISSN = {1631-073X,1778-3569},
   MRCLASS = {15A42 (15A18 46B09 60G70)},
  MRNUMBER = {2441920},
MRREVIEWER = {Werner\ Linde},
       DOI = {10.1016/j.crma.2008.07.011},
       URL = {https://doi.org/10.1016/j.crma.2008.07.011},
}

@article {BR17,
    AUTHOR = {Basak, Anirban and Rudelson, Mark},
     TITLE = {Invertibility of sparse non-{H}ermitian matrices},
   JOURNAL = {Adv. Math.},
  FJOURNAL = {Advances in Mathematics},
    VOLUME = {310},
      YEAR = {2017},
     PAGES = {426--483},
      ISSN = {0001-8708,1090-2082},
   MRCLASS = {60B20 (46B09)},
  MRNUMBER = {3620692},
MRREVIEWER = {Thomas\ Kriecherbauer},
       DOI = {10.1016/j.aim.2017.02.009},
       URL = {https://doi.org/10.1016/j.aim.2017.02.009},
}

@misc{DF24,
  AUTHOR = {Dabagia, Max and Fern{\'a}ndez, Manuel},
  TITLE  = {The smallest singular value of inhomogeneous random rectangular matrices},
  YEAR   = {2024},
  NOTE   = {Preprint, available at \href{https://arxiv.org/abs/2408.14389}{arXiv:2408.14389}},
}

@article {Edelman88,
    AUTHOR = {Edelman, Alan},
     TITLE = {Eigenvalues and condition numbers of random matrices},
   JOURNAL = {SIAM J. Matrix Anal. Appl.},
  FJOURNAL = {SIAM Journal on Matrix Analysis and Applications},
    VOLUME = {9},
      YEAR = {1988},
    NUMBER = {4},
     PAGES = {543--560},
      ISSN = {0895-4798},
   MRCLASS = {15A52 (62H10)},
  MRNUMBER = {964668},
MRREVIEWER = {D.\ S.\ Tracy},
       DOI = {10.1137/0609045},
       URL = {https://doi.org/10.1137/0609045},
}

@article {Edelman91,
    AUTHOR = {Edelman, Alan},
     TITLE = {The distribution and moments of the smallest eigenvalue of a
              random matrix of {W}ishart type},
   JOURNAL = {Linear Algebra Appl.},
  FJOURNAL = {Linear Algebra and its Applications},
    VOLUME = {159},
      YEAR = {1991},
     PAGES = {55--80},
      ISSN = {0024-3795,1873-1856},
   MRCLASS = {62E15 (62H10)},
  MRNUMBER = {1133335},
MRREVIEWER = {P.\ N.\ Rathie},
       DOI = {10.1016/0024-3795(91)90076-9},
       URL = {https://doi.org/10.1016/0024-3795(91)90076-9},
}

@article {GLPTJ17,
    AUTHOR = {Gu\'edon, Olivier and Litvak, Alexander E. and Pajor, Alain
              and Tomczak-Jaegermann, Nicole},
     TITLE = {On the interval of fluctuation of the singular values of
              random matrices},
   JOURNAL = {J. Eur. Math. Soc. (JEMS)},
  FJOURNAL = {Journal of the European Mathematical Society (JEMS)},
    VOLUME = {19},
      YEAR = {2017},
    NUMBER = {5},
     PAGES = {1469--1505},
      ISSN = {1435-9855,1435-9863},
   MRCLASS = {60B20 (15B52 46B09 60E15)},
  MRNUMBER = {3635358},
       DOI = {10.4171/JEMS/697},
       URL = {https://doi.org/10.4171/JEMS/697},
}

@article {LPRT05,
    AUTHOR = {Litvak, A. E. and Pajor, A. and Rudelson, M. and
              Tomczak-Jaegermann, N.},
     TITLE = {Smallest singular value of random matrices and geometry of
              random polytopes},
   JOURNAL = {Adv. Math.},
  FJOURNAL = {Advances in Mathematics},
    VOLUME = {195},
      YEAR = {2005},
    NUMBER = {2},
     PAGES = {491--523},
      ISSN = {0001-8708,1090-2082},
   MRCLASS = {52A22 (15A52 46B09 60D05)},
  MRNUMBER = {2146352},
MRREVIEWER = {B\'ela\ Uhrin},
       DOI = {10.1016/j.aim.2004.08.004},
       URL = {https://doi.org/10.1016/j.aim.2004.08.004},
}

@article {LR12,
    AUTHOR = {Litvak, Alexander E. and Rivasplata, Omar},
     TITLE = {Smallest singular value of sparse random matrices},
   JOURNAL = {Studia Math.},
  FJOURNAL = {Studia Mathematica},
    VOLUME = {212},
      YEAR = {2012},
    NUMBER = {3},
     PAGES = {195--218},
      ISSN = {0039-3223,1730-6337},
   MRCLASS = {60B20 (15B52)},
  MRNUMBER = {3009072},
MRREVIEWER = {Anna\ Lytova},
       DOI = {10.4064/sm212-3-1},
       URL = {https://doi.org/10.4064/sm212-3-1},
}

@article {Livshyts21,
    AUTHOR = {Livshyts, Galyna V.},
     TITLE = {The smallest singular value of heavy-tailed not necessarily
              i.i.d. random matrices via random rounding},
   JOURNAL = {J. Anal. Math.},
  FJOURNAL = {Journal d'Analyse Math\'ematique},
    VOLUME = {145},
      YEAR = {2021},
    NUMBER = {1},
     PAGES = {257--306},
      ISSN = {0021-7670,1565-8538},
   MRCLASS = {60B20},
  MRNUMBER = {4361906},
MRREVIEWER = {Khanh\ Duy\ Trinh},
       DOI = {10.1007/s11854-021-0183-2},
       URL = {https://doi.org/10.1007/s11854-021-0183-2},
}

@article {LTV21,
    AUTHOR = {Livshyts, Galyna V. and Tikhomirov, Konstantin and Vershynin,
              Roman},
     TITLE = {The smallest singular value of inhomogeneous square random
              matrices},
   JOURNAL = {Ann. Probab.},
  FJOURNAL = {The Annals of Probability},
    VOLUME = {49},
      YEAR = {2021},
    NUMBER = {3},
     PAGES = {1286--1309},
      ISSN = {0091-1798,2168-894X},
   MRCLASS = {60B20 (15B52)},
  MRNUMBER = {4255145},
MRREVIEWER = {Asad\ Lodhia},
       DOI = {10.1214/20-aop1481},
       URL = {https://doi.org/10.1214/20-aop1481},
}

@article {RT18,
    AUTHOR = {Rebrova, Elizaveta and Tikhomirov, Konstantin},
     TITLE = {Coverings of random ellipsoids, and invertibility of matrices
              with i.i.d. heavy-tailed entries},
   JOURNAL = {Israel J. Math.},
  FJOURNAL = {Israel Journal of Mathematics},
    VOLUME = {227},
      YEAR = {2018},
    NUMBER = {2},
     PAGES = {507--544},
      ISSN = {0021-2172,1565-8511},
   MRCLASS = {60B20 (15B52 60D05)},
  MRNUMBER = {3846333},
MRREVIEWER = {Yuliy\ M.\ Baryshnikov},
       DOI = {10.1007/s11856-018-1732-y},
       URL = {https://doi.org/10.1007/s11856-018-1732-y},
}

@article {RV08b,
    AUTHOR = {Rudelson, Mark and Vershynin, Roman},
     TITLE = {The least singular value of a random square matrix is
              {$O(n^{-1/2})$}},
   JOURNAL = {C. R. Math. Acad. Sci. Paris},
  FJOURNAL = {Comptes Rendus Math\'ematique. Acad\'emie des Sciences. Paris},
    VOLUME = {346},
      YEAR = {2008},
    NUMBER = {15-16},
     PAGES = {893--896},
      ISSN = {1631-073X,1778-3569},
   MRCLASS = {60G70 (15A18 15A42 15A52 60E15)},
  MRNUMBER = {2441928},
       DOI = {10.1016/j.crma.2008.07.009},
       URL = {https://doi.org/10.1016/j.crma.2008.07.009},
}

@article {RV08,
    AUTHOR = {Rudelson, Mark and Vershynin, Roman},
     TITLE = {The {L}ittlewood-{O}fford problem and invertibility of random
              matrices},
   JOURNAL = {Adv. Math.},
  FJOURNAL = {Advances in Mathematics},
    VOLUME = {218},
      YEAR = {2008},
    NUMBER = {2},
     PAGES = {600--633},
      ISSN = {0001-8708,1090-2082},
   MRCLASS = {60E15 (60B20)},
  MRNUMBER = {2407948},
MRREVIEWER = {Ben\ Joseph\ Green},
       DOI = {10.1016/j.aim.2008.01.010},
       URL = {https://doi.org/10.1016/j.aim.2008.01.010},
}

@article {RV09,
    AUTHOR = {Rudelson, Mark and Vershynin, Roman},
     TITLE = {Smallest singular value of a random rectangular matrix},
   JOURNAL = {Comm. Pure Appl. Math.},
  FJOURNAL = {Communications on Pure and Applied Mathematics},
    VOLUME = {62},
      YEAR = {2009},
    NUMBER = {12},
     PAGES = {1707--1739},
      ISSN = {0010-3640,1097-0312},
   MRCLASS = {60B20 (15A42 15B52 60E15)},
  MRNUMBER = {2569075},
MRREVIEWER = {Mark\ W.\ Meckes},
       DOI = {10.1002/cpa.20294},
       URL = {https://doi.org/10.1002/cpa.20294},
}

@article {RV15,
    AUTHOR = {Rudelson, Mark and Vershynin, Roman},
     TITLE = {Small ball probabilities for linear images of high-dimensional
              distributions},
   JOURNAL = {Int. Math. Res. Not. IMRN},
  FJOURNAL = {International Mathematics Research Notices. IMRN},
      YEAR = {2015},
    NUMBER = {19},
     PAGES = {9594--9617},
      ISSN = {1073-7928,1687-0247},
   MRCLASS = {60E15},
  MRNUMBER = {3431603},
MRREVIEWER = {Mikhail\ A.\ Lifshits},
       DOI = {10.1093/imrn/rnu243},
       URL = {https://doi.org/10.1093/imrn/rnu243},
}

@article {Szarek91,
    AUTHOR = {Szarek, Stanis\l aw J.},
     TITLE = {Condition numbers of random matrices},
   JOURNAL = {J. Complexity},
  FJOURNAL = {Journal of Complexity},
    VOLUME = {7},
      YEAR = {1991},
    NUMBER = {2},
     PAGES = {131--149},
      ISSN = {0885-064X,1090-2708},
   MRCLASS = {65F35 (60F99 68Q25)},
  MRNUMBER = {1108773},
MRREVIEWER = {M.\ Z.\ Nashed},
       DOI = {10.1016/0885-064X(91)90002-F},
       URL = {https://doi.org/10.1016/0885-064X(91)90002-F},
}

@article {TV09,
    AUTHOR = {Tao, Terence and Vu, Van H.},
     TITLE = {Inverse {L}ittlewood-{O}fford theorems and the condition
              number of random discrete matrices},
   JOURNAL = {Ann. of Math. (2)},
  FJOURNAL = {Annals of Mathematics. Second Series},
    VOLUME = {169},
      YEAR = {2009},
    NUMBER = {2},
     PAGES = {595--632},
      ISSN = {0003-486X,1939-8980},
   MRCLASS = {60G50 (15B52 60E15 60F05)},
  MRNUMBER = {2480613},
MRREVIEWER = {Michael\ Stolz},
       DOI = {10.4007/annals.2009.169.595},
       URL = {https://doi.org/10.4007/annals.2009.169.595},
}

@article {Vershynin11,
    AUTHOR = {Vershynin, Roman},
     TITLE = {Spectral norm of products of random and deterministic
              matrices},
   JOURNAL = {Probab. Theory Related Fields},
  FJOURNAL = {Probability Theory and Related Fields},
    VOLUME = {150},
      YEAR = {2011},
    NUMBER = {3-4},
     PAGES = {471--509},
      ISSN = {0178-8051,1432-2064},
   MRCLASS = {60B20 (60E15)},
  MRNUMBER = {2824864},
MRREVIEWER = {Rajat\ Subhra\ Hazra},
       DOI = {10.1007/s00440-010-0281-z},
       URL = {https://doi.org/10.1007/s00440-010-0281-z},
}

@book {Vershynin18,
    AUTHOR = {Vershynin, Roman},
     TITLE = {High-dimensional probability},
    SERIES = {Cambridge Series in Statistical and Probabilistic Mathematics},
    VOLUME = {47},
      NOTE = {An introduction with applications in data science,
              With a foreword by Sara van de Geer},
 PUBLISHER = {Cambridge University Press, Cambridge},
      YEAR = {2018},
     PAGES = {xiv+284},
      ISBN = {978-1-108-41519-4},
   MRCLASS = {60-01 (60B05 60B20 60E15 60Fxx 62H25)},
  MRNUMBER = {3837109},
MRREVIEWER = {Sasha\ Sodin},
       DOI = {10.1017/9781108231596},
       URL = {https://doi.org/10.1017/9781108231596},
}
\end{document}